\documentclass[11pt,reqno]{amsart}
\usepackage{graphicx}
\usepackage{latexsym}
\usepackage{amssymb}
\usepackage{amsmath}
\usepackage{enumerate}

\usepackage{amsmath}
\usepackage{stmaryrd}
\usepackage{hyperref}
\usepackage{amssymb}
\usepackage{CJK}
\usepackage{color}

\usepackage{amsmath,latexsym,amssymb,amsthm,array,amsfonts}
\usepackage{mathtools, stmaryrd}
\usepackage{xparse} \DeclarePairedDelimiterX{\Iintv}[1]{\llbracket}{\rrbracket}{\iintvargs{#1}}
\NewDocumentCommand{\iintvargs}{>{\SplitArgument{1}{,}}m}
{\iintvargsaux#1} %
\NewDocumentCommand{\iintvargsaux}{mm} {#1\mkern1.5mu..\mkern1.5mu#2}
\usepackage{mathrsfs,dsfont,bbm}
\usepackage{csquotes}
\usepackage{layout}
\usepackage{layout}

\newtheorem{lemma}{Lemma}
\newtheorem{definition}{Definition}
\newtheorem{corollary}{Corollary}
\newtheorem{theorem}{Theorem}
\newtheorem{proposition}{Proposition}
\newtheorem{remark}{Remark}

\newtheorem{conjecture}{Conjecture}

\def\real{{\mathord{{\rm I\kern-2.8pt R}}}}        
\def\inte{{\mathord{{\rm I\kern-2.8pt N}}}}

\def\sZZ{{\rm Z\kern-2.8ptem{}Z}}

\def\z{{\mathchoice
  {\sZZ}
  {\sZZ}
  {\rm Z\kern-0.30em{}Z}
  {\rm Z\kern-0.25em{}Z} }}
\def\sQQ{{\kern 0.27em \vrule height1.45ex width0.03em depth0em
          \kern-0.30em \rm Q}}
\def\qu{{\mathchoice
    {\sQQ}
    {\sQQ}
  {\kern 0.225em \vrule height1.05ex width0.025em depth0em \kern-0.25em \rm Q}
  {\kern 0.180em \vrule height0.78ex width0.020em depth0em \kern-0.20em \rm Q}
        }}
\def\sCC{{\kern 0.27em \vrule height1.45ex width0.03em depth0em
          \kern-0.30em \rm C}}
\def\complex{{\mathchoice
    {\sCC}
    {\sCC}
  {\kern 0.225em \vrule height1.05ex width0.025em depth0em \kern-0.25em \rm C}
  {\kern 0.180em \vrule height0.78ex width0.020em depth0em \kern-0.20em \rm C}
        }}

\MakeOuterQuote{"}

\newcommand{\ba}{\begin{array}}
\newcommand{\ea}{\end{array}}
\newcommand{\be}{\begin{equation}}
\newcommand{\ee}{\end{equation}}
\newcommand{\bea}{\begin{eqnarray}}
\newcommand{\eea}{\end{eqnarray}}
\newcommand{\beaa}{\begin{eqnarray*}}
\newcommand{\eeaa}{\end{eqnarray*}}

\def\z{\zeta}

\font\tenmath=msbm10 \font\sevenmath=msbm7 \font\fivemath=msbm5
\newfam\mathfam \textfont\mathfam=\tenmath
\scriptfont\mathfam=\sevenmath \scriptscriptfont\mathfam=\fivemath

\def \={{\buildrel {\rm (law)} \over =}}

\def\qed{ \hfill \vrule width.25cm height.25cm depth0cm\smallskip}

\newcommand{\basa}{\begin{assumption}}
\newcommand{\easa}{\end{assumption}}

\newcommand{\bas}{\begin{assum}}
\newcommand{\eas}{\end{assum}}

\newcommand{\E}{\mathbb{E}}

\usepackage{amsmath,amsthm}
\usepackage{amssymb}
\DeclareMathOperator{\Hess}{Hess}

\DeclareMathOperator{\Ent}{Ent}

\DeclareMathOperator{\Vol}{Vol}

\DeclareMathOperator{\supp}{supp}
\DeclareMathOperator{\argmin}{argmin}
\DeclareMathOperator{\Tr}{Tr}
\DeclareMathOperator{\tr_N}{tr_n}
\DeclareMathOperator{\tr_N}{tr_N}

\DeclareMathOperator{\Id}{Id}
\DeclareMathOperator{\bary}{bar}
\DeclareMathOperator{\diam}{diam}

\newcommand{\ignore}[1]{}
\begin{document}

\title[A sharp symmetrized free transport-entropy inequality]{A sharp symmetrized free transport-entropy inequality for the semicircular law}

\author{Charles-Philippe Diez$^{\dagger}$}

 \thanks{$^{\dagger}$, University of Luxembourg, Department of Mathematics, Maison du Nombre,
6 avenue de la Fonte, L-4364 Esch-sur-Alzette, Grand Duchy of Luxembourg, \textbf{charles-philippe.diez@uni.lu}}

\begin{abstract}
In this paper, following the recent work of Fathi (2018) in the classical case, we provide by two different methods a sharp symmetrized free Talagrand inequality for the semicircular law, which improves the free TCI of Biane and Voiculescu (2000). The first proof holds only in the one-dimensional case and has the advantage of providing a connection with the machinery of free moment maps introduced by Bahr and Boschert (2023) and a free inverse Log-Sobolev inequality. This case also sheds light on a dual formulation via the free version of the functional Blaschke-Santaló inequality. The second proof gives the result in a multidimensional setting and relies on a random matrix approximation approach developed by Biane (2003), Hiai, Petz and Ueda (2004) combined with Fathi's inequality on Euclidean spaces.
\end{abstract}
\maketitle

\section{Introduction}
\begin{flushleft}
The celebrated Talagrand transportation cost inequality with respect to the standard Gaussian measure in $\mathbb{R}^d$ states that for any probability measure $\mu$, we have: 
\begin{equation}\label{TI}
    W_2(\mu,\gamma)^2\leq 2\Ent_{\gamma}(\mu),
\end{equation}
\end{flushleft}
\begin{flushleft}
  where $W_2$ denote the $L^2$ Kantorovitch-Wasserstein distance on the space of probability measures $\mathcal{P}(\mathbb{R}^d)$,
\begin{equation}
        W_2(\mu,\nu)^2:=\inf \left\{\int \lvert x-y\rvert^2d\pi; \pi\in \mathcal{P}(\mathbb{R}^d\times \mathbb{R}^d),\pi(\cdot,\mathbb{R}^d)=\mu,\pi(\mathbb{R}^d,\cdot)=\nu\right\}.\nonumber
    \end{equation}
    and 
    \begin{equation}
        \Ent_{\gamma}(\mu)=\int f\log fd\gamma;\; \nu=f\gamma,\nonumber
    \end{equation}
   denotes the relative entropy (also called Kullback-Leibler divergence) with respect to the Gaussian measure.
\end{flushleft}
\begin{flushleft}
    This transport inequality was first introduced to study the concentration of measure \cite{Talag} and its applications to Banach space geometry. In this direction we have, for example, the important result of Milman \cite{milman}, who used this concentration phenomenon on the sphere to find a new proof of the famous Dvoretzky's theorem \cite{dvoret} on the existence of high-dimensional almost Euclidean sections of convex bodies. We refer to the survey paper by Gozlan \cite{gozlan} for precise connections between the phenomenon of concentration of measure and functional inequalities.
    \bigbreak
    Recently, a sharp version of Talagrand's inequality for the Gaussian measure was discovered by Fathi \cite{FathT}, giving a strong entropy-$W_2$ bound. This powerful idea of dualiazing between Santal\'o-type inequalities and symmetric transport inequalities
transport inequalities was actually considered in an earlier paper by Nathael Gozlan \cite{gozlan2}. This connexion is in fact build upom the equivalence between the symmetrized transport-entropy inequalities and an \textit{even Maurey's} property $(\tau)$ refining the one introduced originally by Maurey \cite{Maurey}. We are very grateful to Max Fathi for providing this historical reference.
\bigbreak
Two different proofs were originally proposed by Fathi: the first is based on connections between optimal transport theory and the calculus of variations, more precisely between the notion of {\it moment measures} (see the seminal work by Cordero-Erausquin and Klartag \cite{mm}), the Kähler-Einstein PDE \cite{KlaK,KlaKol} and the Gaussian-Talagrand inequality. The second shows that this symmetrized Talagrand inequality can be seen as the dual formulation of the Blaschke-Santaló inequality in convex geometry (there are many variants of this inequality; for convenience and because it has the most general form, we will mostly use a functional version due to Lehec \cite{Lehec}). More recently, a purely stochastic proof based on a new coupling induced by time-reversal martingale representations was found by Courtade, Fathi and Mikulincer \cite{stoch}, following Lehec's work \cite{lehec2} on proving various functional inequalities by stochastic methods.
\end{flushleft}
\begin{theorem}(Fathi, theorem 1.1 in \cite{FathT})\label{fat}
    Let $\mu$ be a centered probability measure and $\nu$ another probability measure. Then
    \begin{equation}
        W_2(\mu,\nu)^2\leq 2\Ent_{\gamma}(\mu)+2\Ent_{\gamma}(\nu).
    \end{equation}
Furthermore, equality holds if and if only there exists a symmetric  positive matrix $A$ such that $\mu$ is a non-degenerate Gaussian measure with covariance matrix $A$ and $\nu$ is a Gaussian measure of covariance $A^{-1}$ (not necessarily centered).
\end{theorem}
As a direct corollary (the chain rule for the entropy) when $\gamma$ is replaced by $\gamma_{a^2}$, a centered Gaussian of (isotropic) covariance matrix $\alpha^2\Id>0$, then the following holds:
\begin{corollary}\label{covgau}
    Let $\mu$ be a centered probability measure and $\nu$ another probability measure. Then
    \begin{equation}
        W_2(\mu,\nu)^2\leq \frac{2}{\alpha^2}(\Ent_{\gamma_{a^2\Id}}(\mu)+\Ent_{\gamma_{a^2\Id}}(\nu)).
    \end{equation}
\end{corollary}
Using the Cafarelli contraction theorem \cite{Cafa}, Fathi was able to show that this theorem can even be pushed onto the class of uniformly log-concave probability measures.
\begin{theorem}(Fathi, theorem 1.5 in \cite{FathT}) \label{thfathi}
    Let $\theta=e^{-V}$ a symmetric, uniformly log-concave probability measure, that is the potential $V$ is smooth and satisfies $\Hess V\geq \alpha Id$ for some $\alpha>0$. Then for any symmetric probability measure $\mu$ and any other probability measure $\nu$, we have
    \begin{equation}
        W_2(\mu,\nu)^2\leq \frac{2}{\alpha}(\Ent_{\theta}(\mu)+\Ent_{\theta}(\nu)).
    \end{equation}
    The symmetry assumption on $\mu$ and $\theta$ is not essential and can be relaxed to the following assumption:
denote $T$ the \textit{optimal transport map} sending $\gamma$ onto $\theta$, then the theorem still holds if the pushforward of $\mu$ by $T^{-1}$ is centered, i.e. if the transport map $T$ sending $\mu$ onto $\theta$ is such that $\int Td\mu=0$.
\end{theorem}
\begin{flushleft}
These previous theorems have in fact recently been extended by Tsuji \cite{noncenter} when $\mu$ does not necessarily have barycenter zero, thus giving the full picture of sharp transport cost inequalities in full generality.
\end{flushleft}
\begin{theorem}(Tsuji, corollary 3.2 in \cite{noncenter})
    Let $\mu,\nu$ be probability measures in $\mathcal{P}_2(\mathbb{R}^d)$. Then
    \begin{equation}
        W_2(\mu,\nu)^2\leq 2\Ent_{\gamma}(\mu)+2\Ent_{\gamma}(\nu)-2\langle \bary(\mu),\bary(\nu)\rangle.
    \end{equation}
\end{theorem}
\begin{flushleft}
   In the free context \cite{BV}, using techniques from the pioneering work of Otto and Villani \cite{OV} and, more precisely, a careful analysis of the complex Burger equation, Biane and Voiculescu obtained in a breakthrough work the free analogue of the Talagrand TCI. In fact, a few years after this first important result, first Biane in \cite{BLS} and then Hiai, Petz and Ueda \cite{freeTCI, HiaiUeda} started a (now very successful) programme to derive free functional inequalities (free LSI, free Talagrand, HWI...) using random matrix approximations. 
   \end{flushleft}
   \begin{flushleft}
   This approach is based in particular on the large deviation asymptotics of spectral measures of unitary invariant Hermitian random matrices, as discovered by Voiculescu \cite{V1} (through his microstates approximation) and formulated as a large deviation principle by Ben Arous and Guionnet \cite{BAG} and Hiai, Petz and Ueda \cite{LHP,LE}. For convenience, we will recall a useful (and more general than necessary) result which will be used several times in the sequel.
   \end{flushleft}
\begin{flushleft}
First, note that if a random matrix $X_N$ of size $N\times N$ is distributed according to the law
\begin{equation}
    \mu_N\propto \exp(-N\Tr(V(M)))d\Lambda_N(M),
\end{equation} 
where $\Lambda$ denotes the Lebesgue measure on the space of $N\times N$ Hermitian matrices $M_N^{s.a}$, and $V$ is a potential such that $\mu_N$ is indeed a probability measure on $M_N^{s.a}$, we have that the joint law of its eigenvalues is given by
\end{flushleft}
\begin{equation}
\mathbb{P}_V^N(dx_1,\ldots, dx_N)=\prod_{i<j}\lvert x_i-x_j\rvert^2\exp\bigg(-N\sum_{i=1}^NV(x_i)\bigg)\frac{\prod _{i=1}^Ndx_i}{Z_V^N},
\end{equation}
where $Z_V^N$ is a normalising constant.
\begin{flushleft}
    So if we introduce the empirical measure
    \begin{equation}
        \tilde{\mu}_N:=\frac{1}{N}\sum_{i=1}^N\delta_{x_i},
    \end{equation}
The density of eigenvalues can be rewritten as 
\begin{equation}
    \mathbb{P}_V^N(dx_1,\ldots,dx_N) = \exp(-N^2\tilde{J}_V^N(\tilde{\mu}_N))\frac{\prod_{i=1}^Ndx_i}{Z_V^N}.
\end{equation}
where $J_V(\mu)=\int V(x)d\mu(x)-\iint_{x\neq y}\log\lvert x-y\rvert d\mu(x)d\mu(y)$.
\end{flushleft}
\begin{flushleft}
    The following important theorem, proved by Ben Arous and Guionnet \cite{BAG}, is given in a generalized form (i.e. with a weaker hypothesis) which is satisfied by all the potentials considered in our paper.
\end{flushleft}
\begin{theorem}(Large deviations for the empirical measure, Ben Arous and Guionnet \cite{BAG})
\newline
    Let $V$ be a continuous function satisfying the following hypothesis 
    \begin{equation}
        \underset{\lvert x\rvert \rightarrow \infty}{\liminf}\frac{V(x)}{x^2}>0.
    \end{equation} 
    According to the law $\mu_N\propto \exp(-V(M))d\Lambda_N(M)$, the sequence of random measures $\tilde{\mu}_N$ satisfies a large deviation principle in speed $N^2$ with good rate function $H(\cdot,\nu_V):=\chi_V(\cdot)-\chi_V(\nu_V)$.
\end{theorem}
\begin{flushleft}
This relative free entropy $\chi_V$ for a potential $V$ as well as the free Gibbs measure $\nu_V$, which (if it exists) realises the maximum of a suitable free relative entropy functional, will be properly introduced in the next pages.
\end{flushleft}
\begin{flushleft}
    The following important theorem is the free transport cost inequality for the quadratic Wasserstein distance originally due to Biane and Voiculescu \cite{BV}. This is exactly the free counterpart of \eqref{TI} up to the replacement of the relative Gaussian entropy by its free counterpart, called the free entropy adapted to the free Ornstein-Uhlenbeck process. 
\end{flushleft}
    \begin{theorem}(Biane, Voiculescu \cite{BV}, Hiai, Petz, Ueda \cite{freeTCI}, free Talagrand inequality).
    \newline
        If we denote $\sigma$ as the standard semicircular law, then for any compactly supported probability measure $\mu$ we have
        \begin{equation}
            W_2(\mu,\sigma)^2\leq 2H(\mu,\sigma)).
        \end{equation}
    \end{theorem}
    The right-hand side is called \textit{the free entropy adapted to the free Ornstein-Uhlenbeck process}:
\begin{eqnarray}
    H(\mu,\sigma)&:=&\chi_{\frac{x^2}{2}}(\sigma)-\chi_{\frac{x^2}{2}}(\mu)\nonumber\\
    &=&-\chi(\mu)+\frac{1}{2}\int x^2 d\mu +\chi(\sigma)-\frac{1}{2}\nonumber\\
    &=& \frac{1}{2}\int x^2 d\mu-\chi(\mu)+\frac{1}{2}\log(2\pi),\nonumber
\end{eqnarray}
where $\chi(\mu)$ is the free entropy introduced by Voiculescu in \cite{Ventro} for a given Borel probability measure $\mu$.
\begin{equation}
        \chi(\mu):=\iint \log\lvert t-s\rvert d\mu(s)d\mu(t)+\frac{3}{4}+\frac{1}{2}\log 2\pi.
    \end{equation}
    \begin{flushleft}
    In particular, and up to constants, the free entropy is the minus sign of the logarithmic energy of $\mu$.
\end{flushleft}
 
\begin{flushleft}
The free analogue of the Gibbs measure $e^{-u}dx$ with $u:\mathbb{R}\rightarrow \mathbb{R}\cup\left\{+\infty\right\}$ are called \textit{free Gibbs measure}. They will be referred to as $\nu_u$ (sometimes called free equilibrium measures) and are, as in the classical case, the solution to a variational problem.
\begin{definition}\label{variation}
The free Gibbs measure $\nu_u$ associated with the (convex or not) function $u:\mathbb{R}\rightarrow \mathbb{R}$ is the measure corresponding to the free Gibbs law $\tau_u$. This is the unique maximizer of the functional:
\begin{equation}\label{13}
\chi_u(\mu):=\chi(\mu)-\int u(s)d\mu(s),
\end{equation}
\end{definition}
It is known that the existence of a maximizer is ensured in every dimension when $u$ is a small analytic perturbation of the quadratic potential or for a convexity condition based on the contractivity of a certain semigroup introduced by Dabrowski, Guionnet and Shlyakhtenko \cite{YGS}. In the one-dimensional case it is known (see e.g. \cite{saff}) that the existence and uniqueness of a (compactly supported) maximizer $\nu_u\in \mathcal{P}(\mathbb{R})$ is ensured when $u$ is lower semi-continuous and satisfies some logarithmic growth at infinity:
\begin{equation}\label{gibbs}
    \underset{\lvert x\rvert \rightarrow \infty}{\lim}(u(x)-2\log\lvert x\rvert)=+\infty,
\end{equation}
We can extend this existence more generally for any unbounded set $S$ of the real line and for $u:S\rightarrow \mathbb{R}$ well-defined and satisfying the hypothesis \ref{gibbs}. However, if we work with $u: S\rightarrow \mathbb{R}$, where $S$ is a closed bounded set of the real line, we even have to assume $\mathcal{C}^3$ regularity on the interior of $S$, see Ledoux, Popescu \cite{LP1}.
\bigbreak
The variational characterization of the maximizer $\nu_u$ (see \cite{saff}, theorem $1.3$) also shows that there exists a constant $C\in \mathbb{R}$ (note that we don't insist on the notion of quasi-everywhere equality and refer to \cite{saff} for details),
\begin{eqnarray}
    u(x)&\geq& 2\int \log\lvert x-y\rvert d\nu_u(y),\:\mbox{for quasi-every}\: x\in \mathbb{R}\nonumber\\
    u(x)&=& 2\int \log\lvert x-y\rvert d\nu_u(y),\:\mbox{for quasi-every}\: x\in \supp(\nu_u)
\end{eqnarray}
The uniqueness of the maximizer can also be obtained if $u$ is bounded from below, satisfies a growth condition at infinity, and satisfies a H\"older-type continuous criterion (see \cite{DmV}).
\bigbreak
It was also noticed by Voiculescu (see section 3.7 of \cite{VF}) that the above maximizer $\nu_u$ (here we assume $u$ to be convex) should satisfy the following {\it Euler-Lagrange} equation: 
\begin{equation}\label{HT}
    2\pi H\nu_u=u',\: \nu_u-a.e
\end{equation}
where $Hp$ denotes the Hilbert transform of a probability measure $\rho$ which is defined as:
\begin{equation}
    H\rho(t)=\frac{1}{\pi}PV\int_{\mathbb{R}}\frac{1}{t-x}d\rho(x),
\end{equation}
In fact, since $\nu_u$ maximize $\chi_u(\cdot)$. We have for a test function $f$, and $\delta\in \mathbb{R}$, that
\begin{equation}
    \chi_u((x+\delta f)_{\sharp} \nu_u)\leq \chi_u(\nu_u),
\end{equation}
Then, by considering $\frac{d}{dt}{\chi_u((x+\delta f)_{\sharp}\nu_u)}_{|t=0}$, it is implied that this last equation \eqref{HT} can also be rewritten as a Schwinger-Dyson (type) equation as follows (which is also true in the multidimensional case, using the microstates version of free entropy):
\begin{proposition}(Guionnet and Shlyakhtenko: Theorem 5.1 in \cite{GS}, Fathi and Nelson: Lemma 1.6 in \cite{FN})
The unique maximizer of the previous functional is necessarily a solution of the following Schwinger-Dyson equation:
\begin{equation}
    \int u'(x)f(x)d\nu_u(x)=\iint \frac{f(x)-f(y)}{x-y}d\nu_u(x)d\nu_u(y),
\end{equation}
for all nice test functions $f$.
\end{proposition}
\begin{remark}\label{lebesgue}
    It is also worth noting that at the maximiser the function $\chi_u(\nu_u)$ is invariant by shifting $u$, i.e. for any $z\in \mathbb{R}$, $\chi_{u_z}(\nu_{u_z})=\chi_u(\nu_u)$, where $z\in \mathbb{R}$ and $u_z(\cdot):=u(z+\cdot)$ is any suitable function satisfying \eqref{gibbs}.
\end{remark}
The cyclic derivative in this case coincides with the usual derivative, i.e. $\mathcal{D}u(x)=u'(x)$, and the non-commutative Jacobian is the operator $\mathcal{J}\mathcal{D}$.
\end{flushleft}
\begin{flushleft}
As it will be needed in the sequel, we also have a sort of reminder on the Legendre transforms and their main properties. We recall that it is an order-reversing involution when restricted to convex functions.
\newline
For $\varphi:\mathbb{R}^d\rightarrow \mathbb{R}\cup\left\{+\infty\right\}$ a function, convex or not, and not identically ${+\infty}$, we denote its Legendre transform as the convex function:
\begin{equation}
\varphi^*(y)=\sup_{\substack{x\in \mathbb{R}^d\\ \varphi(x)<\infty}}\left\{x\cdot y-\varphi(x)\right\}
,
\end{equation}
And which satisfies the following property:
\begin{enumerate}
\item $\varphi^*$ is always convex and lower semi-continuous.
\item when $\varphi$ is convex and differentiable at the point $x$, we have:
\begin{equation}\label{fundam}
    \varphi^*(\nabla \varphi(x))=x\cdot \nabla \varphi(x)-\varphi(x),\nonumber
\end{equation}

\item If $\varphi$ is $\mathcal{C}^2$, $\varphi^*$ is also $\mathcal{C}^2$, and $\nabla \varphi^*$ is the inverse of $\nabla \varphi$, i.e: $\nabla \varphi^*=(\nabla \varphi)^{-1}$.
\item $\varphi(x)+\varphi^*(y)\geq x\cdot y$, for all $x,y \in \mathbb{R}^d$.\label{property}
\end{enumerate}
The unique fixed point of the Legendre transform is given by $\frac{{\lVert x\rVert_2}^{2}}{2}$, i.e. the Legendre transform is a symmetry on the space of convex functions around this fixed point. The geometrical analog of the Legendre transform being the polar transformation, the previous statement can also be translated in terms of convex bodies $K\subset \mathbb{R}^d$ by saying that the unique fixed point of the polar transformation, i.e. $K\mapsto K^{\circ}$ (where $K^{\circ}$ denotes its polar body) is exactly the Euclidean unit ball $B_2^d$ w.r.t. the standard inner product on $\mathbb{R}^d$.
\end{flushleft}
\begin{remark}\label{rem}
Guionnet and Maurel-Ségala in \cite{GM} have indeed proved the equivalence between the Schwinger-Dyson equation and the above maximisation problem, provided that the free Gibbs measure $\nu_u$ exists, is unique and has a connected support. For example:
\begin{enumerate}
    \item If $u$ is strictly convex on a sufficiently large interval, i.e. $\exists \kappa>0$, such that $u''(x)\geq \kappa $ for all $x\in \mathbb{R}$. This also ensures that $\nu_u$ is supported on a compact interval and has a density w.r.t. the Lebesgue measure.
    \item The other case where existence and uniqueness (even in the multidimensional setting) are ensured is when we consider a "small" analytic perturbation of the semicircular potential, i.e. $\frac{1}{2}\lVert x\rVert^2+\beta W$ ($\beta$ small and $W$ a n.c. power series see e.g. \cite{GS}).
\end{enumerate}
\end{remark}
\begin{remark}
    It has been proved by various authors that the free Talgarand inequality for the free quadratic cost ($T_2$) holds for more general convex potentials than the semicircular ones by the seminal work of Hiai, Petz and Ueda \cite{freeTCI}. However (and surprisingly) Maurel-Segala and Maïda \cite{Maid} have shown that free transport entropy inequalities can hold in non-convex situations at the cost of weakening the distance $W_2$ to $W_1$, in particular giving a free analogue of Bobkov and G\"otze's result \cite{bobkov}.
\end{remark}
 \begin{flushleft}
     Our purpose here is to strengthen the previous result of Biane Voiculescu \cite{BV}, Hiai, Petz and Ueda \cite{freeTCI,HiaiUeda} by giving a free analogue of Fathi's sharp symmetrized Talagrand inequality (abbreviated SSTI) which holds in all dimensions for the microstates variant. In dimension one, we can even give a better picture of this transport cost inequality by proving some connection with (its possible) dual formulation via a \it{free Santaló inequality}.
 \end{flushleft}
 \begin{flushleft}
In the classical case, the correspondence between convex functions and Borel measures in $\mathbb{R}^d$ admits of several different views: the Minkowski problem, a vision coming from complex geometry (\textit{toric Kähler manifolds}), or a more probabilistic one via the Monge-Ampère equation (PDE theory) and the transport of measures. 
\newline
The duality between the transport entropy inequality and the functional Blascke-Santaló inequality for the class of $s$-concave functions has recently been further explored by Fradelizi, Gozlan, Sadovsky and Zugmeyer, see section $3$ and more precisely Theorem $2.7$ in \cite{frade}. This allows, in particular, to obtain some optimal transport entropy inequalities for spherically invariant probability models beyond the Gaussian case. Motivated by the barycenter problem in optimal transport theory, Nakamura and Tsuji in \cite{tsuji2} also recently extended the Blaschke-Santal\'o inequality for the case of multiple even functions conjectured by Kolesnikov and Werner \cite{Koles} (who proved the unconditional version of the theorem in their paper) and proved a Talagrand type inequality for multiple even probability measures involving the \textit{Wasserstein barycenter} of these measures. This notion of \textit{Wasserstein barycenter} was introduced by Agueh–Carlier in \cite{AC} and consists an extension of McCann's interpolation \cite{Mccann} to the case of more than two measures. This last important statement has profound consequences in the theory of optimal transport and in fact traduce a specific property of the $L^2$ Wasserstein space: $(\mathcal{P}_2(\mathbb{R}^d),W_2)$ is a Hamadard and a fortiori a geodesic space (uniquely geodesic only for a pair of absolutely continuous probability measures), which leads to a notion of midpoint and to the famous displacement of convexity of functionals in the Wasserstein space. \end{flushleft}
\bigbreak
\begin{flushleft}
    For a little context, we recall the general breakthrough proved by Cordero-Erausquin and Klartag in \cite{CEK}, which is a variant of the optimal transport problem. In this case, the Brenier map is the solution of a variant of the Monge-Amp\`ere equation, called the toric Kähler-Einstein equation. The original approach they used is based on a variational formulation in the class of convex functions, and in particular that the Prékopa-Leindler inequality implies that, for a (nice) measure $\mu$, we have that $J(u):=\log\int_{\mathbb{R}^d} e^{-u^*}dx-\int_{\mathbb{R}^d}fd \mu$ is a concave functional over the space of convex functions finite in a neighborhood of the origin. We will come back to this convex formulation and how to adapt this method in the free case in another paper.
\end{flushleft}

\begin{theorem} (Cordero-Erausquin and Klartag in \cite{CEK}) \label{th3}
Let $\mu$ be a measure in $\mathcal{P}_1(\mathbb{R}^d)$ with barycenter at the origin, and not supported on an hyperplane. Then there exists an essentially-continuous convex function $\varphi$ (uniquely determined up to translations), such that $\mu$ is the pushforward of the centered probability measure with density $e^{-\varphi}$ by the map $\nabla \varphi$. The function $\varphi$ is called the moment map of $\mu$.
\end{theorem}
\begin{flushleft}
In our context, Bahr and Boschert \cite{BB} were able to prove a complete free analogue in the one-dimensional case of this construction by adapting the variational method of Santambrogio \cite{SA} for the classical moment measure. Note also that the condition that the centered measure $\mu$ is not supported on a hyperplane is reduced in this one-dimensional case to the condition: $\mu\neq \delta_0$.
\end{flushleft}
\begin{theorem}(Bahr and Boschert, Theorem 2.5 in \cite{BB})\label{2.5}
Let $\mu\neq \delta_0$, a probability measure in $\mathcal{P}_2(\mathbb{R})$ with barycenter zero, then there exists a convex, lower semi-continuous function $u$, such that $\mu$ is the pushforward of the free Gibbs measure $\nu_u$ by the function $u'$, i.e $\mu=(u')_{\sharp} \nu_u$. The convex function $u$ is called the {\it free moment map} of $\mu$.
\newline
Moreover, $\nu_u$ is absolutely continuous w.r.t the Lebesgue measure, has compact support, and is the unique centered maximizer of the functional 
\begin{equation}
    \mathcal{F}(\rho)=L(\rho)+T(\rho,\mu)
\end{equation}
defined on $\mathcal{P}_2(\mathbb{R})$, where $L(\rho)$ denote the logarithmic energy of $\rho$ and
$T(\rho,\mu)$ is the maximal correlation functional:
\begin{equation}
    T(\rho,\mu)=\sup\left\{\int x\cdot y\: d\gamma, \gamma \in d\pi(\rho,\mu)\right\}
\end{equation}
where $\pi(\rho,\mu)$ is the set of transport plans with marginals $\rho$ and $\mu$.
\end{theorem}
\begin{remark}
    When the free moment measure $\mu$ is itself a centered free Gibbs measure, then the function $\varphi$ is solution to a free \textit{Kahler-Einstein} equation which is a variant of the free Monge-Ampère equation of Guionnet and Shlyakhtenko \cite{GS} (see section $4$ in \cite{CD_mm}).
\end{remark}

\section{A sharp symmetrized one dimensional free TCI}
Although we can derive the free sharp symmetrized Talagrand inequality for the semicircular law using Fathi's inequality on Euclidean spaces \ref{fat} and the random matrix approach used by Hiai, Petz and Ueda in \cite{freeTCI} (which will be done in the section \ref{last}), we prefer to propose first in the one-dimensional case a proof based on the notion of \textit{free moment measure} introduced by Bahr and Boschert \cite{BB}. This has the advantage of providing interesting analogies with the theory of optimal transport and the calculus of variations, and in particular with the machinery of \textit{free moment maps}. For a bit of context, and because we are going to relate this sharp Talagrand inequality to optimal transport theory, we need to introduce a fundamental result, which is a reversed Log-Sobolev inequality for the class of log-concave measures.
\begin{theorem}(Inverse Log-Sobolev inequality)\label{rls}
    Let $\rho=e^{-\varphi}dx$ be a log-concave probability measure in $\mathbb{R}^d$, and define its Shannon-Boltzmann entropy as $S(\rho):=\int \varphi d\rho$, then we have:
    \begin{equation}
        S(\gamma)-S(\rho)\geq \frac{1}{2}\int \log\det(\Hess \varphi)d\rho,
    \end{equation}
Moreover, equality holds if and if only $\rho$ is Gaussian with some positive definite covariance matrix $A$.
\end{theorem}

 \begin{flushleft}
This important inequality was first proved by Artstein-Avidan \textit{et al.} \cite{AA} using affine isoperimetric inequalities, but leaving the cases of equality open. The proof is rather technical, and the tools used seem inaccessible in the free context. Recently, a relatively simple and short proof including the equality cases has been given by Caglar et al. \cite{Caglar} using the functional Santaló inequality.
\end{flushleft}
\bigbreak
\begin{flushleft}
Using this classical inverse Log-Sobolev inequality in combination with a random matrix approximation procedure with concentration of measures arguments from Guionnet's seminal work \cite{guionn}, we are indeed able to derive the free version of the inverse Log-Sobolev inequality. Surprisingly, we can also obtain an analytical proof of our result using a free version of the Blaschke-Santal\'o inequality (the latter will be proved in Section $3$), theorem \ref{san}. This was not expected at first sight and relies mainly on the variational characterisation of the free pressure (corresponding to the free version of the log-Laplace functional) and on the properties of the Legendre transform.
\end{flushleft}
\subsection{Random matrix heuristics}
\begin{flushleft}
First, we derive the following free inverse log-Sobolev inequality purely by large deviations for random matrices and a correct understanding of how the rescaled matrix entropies converge to their free counterparts. 
        \end{flushleft}
\begin{flushleft}
   In the following we denote $\rho:=\nu_u$ as a free (and thus compactly supported) Gibbs measure associated with a convex lower semi-continuous potential $u$ satisfying $\underset{\lvert x\rvert\rightarrow +\infty}{\lim}u(x)=+\infty$ and we also denote $\sigma$ as the standard semicircular law.
\end{flushleft}
            Let $M_n^{sa}$ be the set of Hermitian $n\times n$ matrices with inner product according to the Hilbert-Schmidt norm $\Vert A\rVert_{HS}^2: =\Tr_N(A^2)$ (where $\Tr_N$ denotes the non-normalised trace), and the associated Lebesgue measure $d\Lambda_N$ on $M_N^{sa}\simeq \mathbb{R}^{n^2}$. We denote the normalised trace as $\tr_N=\frac{1}{N}\Tr_N$.
            \bigbreak
            Now, we consider the two following measures on $M_N^{sa}$:
            \begin{equation}
                d\sigma_{N}(M)=\frac{1}{Z_N}\exp\bigg(-\frac{N^2}{2}\tr_N(M^2)\bigg)d\Lambda_N(M),\nonumber
            \end{equation}
            \begin{equation}
                 d\rho_{N}(M)=\frac{1}{Z_N^{u}}\exp\bigg(-N^2\tr_N(u(M))\bigg)d\Lambda_N(M)\nonumber
            \end{equation}
            where $Z_N$ and $Z_N^u$ denotes some normalisation constants.
            \bigbreak
            The two measures are well-defined since both $x\mapsto \frac{x^2}{2}$ and $x\mapsto u(x)$ are convex functions going to infinity at infinity, both admits a minimum, so that $u(x)\geq a\lvert x\rvert+b$ for some $a>0,\:b \in \mathbb{R}$. This ensures that both satisfy the logarithmic growth assumptions at infinity \eqref{gibbs}. Recall also that the first one is the standard Gaussian measure on the space of  Hermitian $N\times N$ matrices.
    \bigbreak 
          
            We then denote  $S(\sigma_N)$ and $S(\rho_N)$ respectively as the Shannon-Boltzmann entropy of $\sigma_N$ and $\rho_N$. We then have by definition of the free entropy (see e.g theorem 4.2 in \cite{Hiai1} or section $1$, $(IV)$ in \cite{Hiai2}),
            \begin{equation}
                \chi(\sigma)=\underset{N\rightarrow \infty}{\lim}\bigg(\frac{1}{N^2}S(\sigma_N)+\frac{1}{2}\log N\bigg),\nonumber
            \end{equation}
            and 
              \begin{equation}
                \chi(\rho)=\underset{N\rightarrow \infty}{\lim}\bigg(\frac{1}{N^2}S(\rho_N)+\frac{1}{2}\log N\bigg)\nonumber
                \end{equation}
                And since $\forall N\geq1$, $\rho_{N}$ has a log-concave density w.r.t the Lebesgue measure given by:
                \begin{equation}
                    u_N(M)=N^2\tr_N(u((M)),\nonumber
                \end{equation}
                Thus, by applying the classical inverse Log-Sobolev inequality \ref{rls}, we get that:
                \begin{eqnarray}\label{trlog}
                    \frac{1}{N^2}[S(\sigma_N)-S(\rho_N)]\geq \frac{1}{2{N^2}}\E_{\rho_N}\bigg[\Tr\log \Hess u_N\bigg],
                \end{eqnarray}
                Then, we see that the left-hand side converges to the difference of the free entropies:
                \begin{equation}
                     \frac{1}{N^2}[S(\sigma_N)-S(\rho_N)]\rightarrow \chi(\sigma)-\chi(\rho),\nonumber
                \end{equation}
                \begin{flushleft}
                We have now to distinguish two different cases:
                \begin{enumerate}
        \item We first suppose that $u$ is twice continuously differentiable and strictly convex. Hence, from Guionnet's concentration result (see \cite{guionn}), since $\rho^{(n)}$ has a strictly log-concave density), we get that:
                \begin{equation}
                    \frac{1}{N^2}\E_{\rho_N}\bigg[\Tr\log \Hess u_N\bigg]\rightarrow \rho\otimes \rho (\log \mathcal{J}{\mathcal{D}}u).\nonumber
                \end{equation}
                from which we obtain the desired inequality.
                \item 
                    If it's not the case, we use a regularization. This can be done for example in two steps. 
                    For example, if the function $u$ lacks smoothness, we can use Moreau-Yoshida regularisation which is defined for proper convex function $u:\mathbb{R}\rightarrow \mathbb{R}$  as
                    $u_{\lambda}(x):=\underset{y\in \mathbb{R}}{\min}\left\{u(y)+\frac{1}{2\lambda }\lvert y-x\rvert^2\right\}$
                    as in \cite{delalande} (see lemma 2.6 and proposition 2.5 for various properties of this regularization), which can be used in order to gain smoothness on the potential and preserve its convexity.
                    \bigbreak
                    Next, we add a small quadratic perturbation, i.e for $\epsilon >0$ "small" enough, the function $\tilde{u}_{\epsilon}(x):=u(x)+\frac{\varepsilon}{2}x^2$ which is strictly convex. 
                    It is then easy to see that the free Gibbs measure which is associated with $\tilde{u}_{\epsilon}$ is $\rho_{\epsilon}=\rho\boxplus \eta_{\epsilon}$, the free convolution of $\rho$ with a semicircular distribution of variance $\epsilon$ (see e.g Jekel \cite{jekel}).
                    \newline
                    Hence we get by applying the result in the previous case:
                    \begin{eqnarray}
                    \chi(\sigma)-\chi(\rho_{\epsilon})&\geq& \frac{1}{2}\rho_{\epsilon}\otimes\rho_{\epsilon} (\log \mathcal{J}\mathcal{D}\tilde{u}_{\epsilon})\nonumber\\
                    &=& \frac{1}{2}\rho_{\epsilon}\otimes\rho_{\epsilon} (\log (\mathcal{J}\mathcal{D}u+\epsilon))\nonumber\\
                    &\geq & \frac{1}{2}\max\bigg(\rho_{\epsilon}\otimes\rho_{\epsilon} (\log \mathcal{J}\mathcal{D}u),\log \epsilon\bigg)
                    \end{eqnarray}
                since $x\mapsto log(x)$ is non-decreasing.
                \newline
                But, we know from Voiculescu's paper, proposition $7.5$ in \cite{V} that $\epsilon \mapsto \chi(\rho_{\epsilon})$ is increasing. Its also not difficult to show that the same holds for the \textit{log-Jacobian} part: $\epsilon \mapsto \rho_{\epsilon}\otimes\rho_{\epsilon} (\log \mathcal{J}\mathcal{D}u)$ since $u$ is convex.
                Thus, we get at the limit,
                    \newline
                \begin{equation}
                    \chi(\sigma)-\chi(\rho)\geq \frac{1}{2}\rho\otimes\rho(\log \mathcal{J}\mathcal{D}u)
                \end{equation}
                which leads to the following possible free inverse log-Sobolev inequality.
                 \end{enumerate}
                 \end{flushleft}
\begin{flushleft}
                More interestingly, the following proof shows deep similarities with the original proof of the inverse Log-Sobolev inequality by Caglar et al. \cite{Caglar} using the functional Santaló inequality. This was the connection we were looking for, as this second proof confirms the correct form of this free inverse inequality. The main ingredients of the proof are the variational characterisation of the free pressure defined in section $\ref{sec3}$, the definition of $\ref{press}$, the free Blaschke-Santal\'o inequality, and the geometric properties of the Legendre transform.
                \end{flushleft}
\begin{theorem}(Inverse free log-Sobolev inequality)\label{conj1}
\newline
    Let's denote $\rho:=\nu_u$ be a free Gibbs measure with finite first moment (which implies that moment of any order exists) and associated with a convex lower semi-continuous potential $u$ satisfying $\underset{\lvert x\rvert\rightarrow +\infty}{\lim}u(x)=+\infty$ and $\sigma$ as the standard semicircular law. Then,
    \begin{equation}\label{2.2}
        \chi(\sigma)-\chi(\rho)\geq \frac{1}{2}\rho\otimes\rho (\log \mathcal{J}\mathcal{D}u),
    \end{equation}
    where we denote $
    \rho\otimes\rho (\log \mathcal{J}\mathcal{D}u):=\iint \log \bigg( \frac{u'(x)-u'(y)}{x-y}\bigg) d\rho(x)d\rho(y)$.
    \newline
    Equality occurs in \eqref{2.2} if and if only $\rho$ follows a (non-degenerate) centered semicircular distribution.
        \end{theorem}
        \begin{proof}
        Since there is nothing to prove in the case $\chi(\rho)=-\infty$, we can assume that $\chi(\rho)>-\infty$.
        \bigbreak
         Let us start with the analytic proof using the free form of the Blaschke-Santal\'o inequality, which we will prove in Section $3$, Theorem \ref{san}. In the following, we recall that we denote for $\mu\in \mathcal{M}_c:=\bigcup_{R>0} \mathcal{M}([-R,R])$ (the space of compactly supported probability measures) and $h\in L^1(\mu)$, the pairing $\mu(h)$ as
    \begin{equation}
     \mu(h):=\int h(x)d\mu(x).
    \end{equation}
    Since both terms of the inequality \eqref{2.2} that we aim to prove are invariant under the translations of the free Gibbs measure $\rho$, we can assume without loss of generality that $\rho$ is centered.
    \bigbreak
 Let's assume for now that $u$ is smooth enough, i.e. we assume that $u$ is $\mathcal{C}^2$ and strictly convex over the entire real line. We also denote the following pushforward measure $\mu:=u'_{\sharp}\nu_u=u'_{\sharp}\rho$, which actually makes $\mu$ the moment measure associated with the convex potential $u$ (i.e. the free moment map of $\mu$).
    \bigbreak
    By definition of the free relative entropy or equivalently the free pressure (just apply the maximum characterization of \eqref{variation} or equivalently \eqref{press} for $\rho$ which is the free Gibbs measure potential $u$ and $\mu$ which has a relative entropy with respect potential $u^*$ less than its maximizer $\nu_{u^*}$) we have
        $$\left\{
    \begin{array}{ll}
        -\rho(u) +\chi(\rho)=\chi_u(\nu_u)\\
        -\mu(u^*) +\chi(\mu)\leq \chi_{u^*}(\nu_{u^*})
    \end{array}
\right.
$$
Summing these two inequalities and using the following inequality
for the Legendre transform
\newline
 $\forall x,y\in \mathbb{R},\:u(x)+u^*(y)\geq xy$,
and, since the application conditions are satisfied, using the Blaschke-Santal\'o free inequality, we obtain using \ref{san},
\begin{equation}\label{equal}
\chi(\rho)+\chi(\mu)-\rho(u)-\mu(u^*)\leq \chi_u(\nu_u)+\chi_{u^*}(\nu_{u^*})\leq \log(2\pi).
\end{equation}
Now using Voiculescu's change of variables for the free entropy \cite{Ventro} (we are in a very smooth context here to apply this formula), i.e.
$\chi(\mu)=\chi(u'_{\sharp}\rho)=\chi(\rho)+\rho\otimes \rho(\log \mathcal{J}\mathcal{D}u)$, we get,
\begin{equation}\label{midsan}
\chi(\rho)+\chi(\rho)+\rho\otimes \rho(\log \mathcal{J}\mathcal{D}u)-\int u(x)d\rho(x)-\int u^*(x)d\mu(x)\leq \log(2\pi),
\end{equation}
But since $\mu$ is the pushforward of $\rho$ by $u'$, we have, using the fundamental property of the Legendre transform \eqref{fundam},
\begin{equation}
\int u^*(x)d\mu(x)=\int u^*(u'(x))d\rho(x)=\int xu'(x)d\rho(x)-\int u(x)d\rho(x).
\end{equation}
Hence, we get
\begin{align}
&2\chi(\rho)+\rho\otimes \rho(\log \mathcal{J}\mathcal{D}u)-\int u(x)d\rho(x)-\int u^*(u'(x))d\rho(x)\\=&2\chi(\rho)+\rho\otimes \rho(\log \mathcal{J}\mathcal{D}u)-\int xu'(x)d\rho(x)\leq \log(2\pi).
\nonumber\end{align}
But now a standard application of the \textit{Euler-Lagrange} equation 
\begin{equation}
    2\pi H\rho=u',\: \rho-a.e\nonumber
\end{equation}
implies the following Schwinger-Dyson equation (well known in the free context, this is just a particular application of the \textit{integration by parts} formula for free Gibbs measures)
\begin{equation}
\int xu'(x)d\rho(x)=\iint d\rho(x)d\rho(y)=1.\nonumber
\end{equation}
This integration by parts now gives us that
\begin{equation}
    2\chi(\rho)+\rho\otimes \rho(log \mathcal{J}\mathcal{D}u)-1\leq \log(2\pi).
\end{equation}
But, since $\chi(\sigma)=\frac{1}{2}\log(2\pi e)=\frac{1}{2}+\frac{1}{2}\log(2\pi)$, we reach finally the conclusion.
\begin{flushleft}
The result remains true for a not necessarily $\mathcal{C}^2$ and strictly convex function, but it involves working (especially to make sense of the integrals used above) on the interior of the domain of $u$ and $u^*$, denoted respectively by $\Omega_{u}: =(\left\{x\in \mathbb{R}, u(x)< \infty\right\})^{\mathrm{o}}$ and $\Omega_{u^*}$, since we also remember that $u'$ exists almost everywhere by Rademacher's theorem and a theorem of Alexandrov \cite{Alex}, Busemman and Feller \cite{bus} claim that $u''$ exists almost everywhere inside the convex domain of $u$. The ideas of the proof in this less regular case are just a translation of those of Caglar et al. \cite{Caglar} in the free context (see the proof of their theorem $1$). We leave the details to the reader.
\end{flushleft}
\bigbreak
\begin{flushleft}
    The cases of equality are exactly the cases of equality in the free Blaschke-Santal\'o inequality theorem \ref{san}, which tells us that $\rho$ is semicircular. A standard calculation then confirms that this is a case of equality. 
\end{flushleft}
\end{proof} \qed

        \begin{remark}
                Heuristically, the cases of equality can also be understood again via the random matrix approximation.
                In fact, for the equality to hold, it is necessary that the inequality \eqref{trlog} is actually an equality. This is known to hold (from the equality in the classical inverse logarithmic Sobolev inequality) if for all $n\geq 1$ the measures $\rho^{(n)}$ are themselves Gaussian, from which we deduce that the limiting equilibrium measure $\rho$ then follows a (non-degenerate) semicircular distribution. 
                \begin{flushleft}
                    Note also that $\rho$ is assumed to have finite entropy, and thus a fortiori with the free entropy dimension equal to one. That is, $\delta(\rho)=1+\underset{\epsilon \rightarrow 0}{\limsup} \frac{\chi(\rho_{\epsilon})}{\lvert \log \epsilon\rvert}=1$, since $\rho$ is non-atomic in dimension one according to Voiculescu's formula (see \cite{Ventro}). 
                \end{flushleft}
                
                \end{remark}    \begin{remark}
   The above lemma still holds if the free entropy $\chi$ on the left is replaced by the negative of the logarithmic energy $L$ (the appearing constants vanish).
    \end{remark}

    \begin{flushleft}
    Moreover, the assumption that $u$ is convex {\it l.s.c} with $\underset{\lvert x\rvert\rightarrow +\infty}{\lim}u(x)=+\infty$ gives that $u$ has a global minimum (so that $u(x)\geq a\lvert x\lvert+b$ for some $a>0$ and real $b$), and the finiteness of the first moment, by Proposition $2. 4$ in \cite{BB} implies that $\rho$ has a compact support, is absolutely continuous with respect to the Lebesgue measure, and satisfies the {\it Euler-Lagrange} equation $2\pi H\rho=u'$ on the support of $\rho$.
    \end{flushleft}
 \begin{flushleft}
 Now, as Santambrogio noticed in the classical case, it is also easy to see that the minimisation problem of theorem \ref{2.5} can be reformulated in terms of $W_2$ instead of using the {\it maximal correlation functional} (the second moment of measure $\mu$ being fixed). In this way, the relationship between the moment of measure and the optimal transport becomes clear.
\begin{flushleft}
Therefore the measure $\rho:=\nu_u$, for any convex function $u$ (which is the free moment map, up to translations), satisfies the solution of theorem \ref{2.5}:
\begin{equation}\label{2.4}
    \rho=\argmin \left\{\chi_{\frac{x^2}{2}}(\nu)-\frac{1}{2}W_2(\nu,\mu)^2, \nu \in \mathcal{P}_2(\mathbb{R})\right\},
\end{equation}where 
$\chi_{\frac{x^2}{2}}(\nu):=\frac{1}{2}\int x^2d\nu(x)-\chi(\nu)  $ is the free entropy relative to the standard semicircular potential.
\end{flushleft}
\begin{remark}
In particular, this result includes the free Talagrand inequality of Biane and Voiculescu \cite{BV}, which is obtained by taking $\nu$ as the standard semicircular distribution.
\end{remark}
More interestingly, if $\rho$ is the maximizer in \eqref{2.4}, then the quantity to be minimized above can be rewritten as
\begin{equation}
   -\frac{1}{2} \int \lvert u'(x)-x\rvert^2 d\nu_u(x)+\chi_{\frac{x^2}{2}}(\nu_u),
\end{equation}
 where the first term is the negative of the relative free Fisher information of $\rho$ with respect to the semicircular potential. Thus, this quantity (to a constant) is the inverse of the deficit in the semicircular logarithmic Sobolev inequality, and gives a strong indication that it could be used to derive improvements of free functional inequalities.
\end{flushleft}
\bigbreak
Now, since $\chi(\mu)=\chi(\rho)+\rho\otimes\rho (\log \mathcal{J}\mathcal{D}u)$ (assuming finite entropy for $\rho$), which is the change of variable formula for the free entropy which is valid since (modulo constant appearing in the free entropy), by definition, the minus logarithmic energy satisfies,
\begin{eqnarray}
    -L(\mu)&:=&\iint \log\lvert x-y\rvert d\mu(x)d\mu(y)\nonumber\\
    &=& \iint \log \lvert u'(x)-u'(y)\rvert d\rho(x)d\rho(y)\nonumber\\
    &=&\iint \log \bigg\lvert \frac{u'(x)-u'(y)}{x-y}\cdot (x-y)\bigg\rvert d\rho(x)d\rho(y)\nonumber\\
    &=& \iint  \log \lvert x-y\rvert d\rho(x)d\rho(y)+ \iint \log \bigg( \frac{u'(x)-u'(y)}{x-y}\bigg) d\rho(x)d\rho(y)\nonumber\\
    &=& -L(\rho)+\rho\otimes \rho(\mathcal{J}\mathcal{D}u),
\end{eqnarray}
where we used in the second line that $\mu$ is the pushforward of $\rho$ by $u'$, in the third line that $\rho$ is absolutely continuous w.r.t the Lebesgue measure (a fortiori non-atomic) making the decomposition licit, and finally in the last equality that $\mathcal{J}\mathcal{D}u\geq 0$ since $u$ is convex.
\bigbreak
We then get by applying theorem \ref{rls}:
\begin{equation}\label{santalo}
    2\chi(\sigma)-\chi(\mu)-\chi(\rho)\geq 0,
\end{equation}
Hence, 
the previous inequality \eqref{santalo} is equivalent to:
\begin{eqnarray}
    H(\mu,\sigma)+H(\rho,\sigma)&\geq& \frac{1}{2}\int x^2 d\rho +\frac{1}{2}\int x^2 d\mu -1
    \nonumber\\
    &\geq & \frac{1}{2}\int x^2 d\rho +\frac{1}{2}\int \lvert u'(x)\rvert^2 d\rho -\int x\cdot u'(x)d\rho\nonumber\\
    &\geq & \frac{1}{2}\int \lvert x-u'(x)\rvert^2d\rho=W_2(\mu,\rho)^2,\nonumber
\end{eqnarray}
since for a convex lower-semi-continuous function $u$, we have $\int x\cdot u'(x)d\rho\leq 1$ (and equality when $u$ is smooth, from the Schwinger-Dyson equation for $\rho:=\nu_u$). The last equality being obtained since $u'$ is the optimal transport map.
\begin{flushleft}
Hence we have:
\begin{equation}
    H(\rho,\sigma)-\frac{1}{2}W_2(\mu,\rho)^2\geq -H(\mu,\sigma),\nonumber
\end{equation}
Now, theorem \ref{2.5} asserts that $\rho$ is a maximizer of the functional 
\begin{equation}
       \nu \mapsto H(\mu,\nu)-\frac{1}{2}W_2(\mu,\nu)^2,
\end{equation}
from which we deduce that for any compactly supported probability measure $\nu$
\begin{equation}
    H(\nu,\sigma)-\frac{1}{2}W_2(\mu,\nu)^2\geq -H(\mu,\sigma),\nonumber
\end{equation}
\end{flushleft}
which is equivalent to:
\begin{theorem}\label{th8}
    Let $\mu$ be a centered compactly supported probability measure, and $\nu$ another compactly supported probability measure. Then, we have
    \begin{equation}\label{TalagS}
    W_2(\mu,\nu)^2\leq  2H(\mu,\sigma)+2H(\nu,\sigma),
 \end{equation}
 Equality occurs if and only if $\mu$ follows a semicircular distribution with variance $\sigma^2$ and $\nu$ follows a semicircular distribution with variance $\sigma^{-2}$.
\end{theorem}

\begin{flushleft}
    To understand the case of equality, it suffices to see that equality holds if and only if equality occurs in the theorem \ref{conj1}, in which case $\rho$ must follow a semicircular distribution with variance $\sigma>0$. Hence $\mu$, which is the projection of $\rho$ by $u'$ (a linear function in this case), also follows a semicircular distribution with variance $\sigma^{-1}$. Finally, a standard calculation confirms that this is a case of equality.
\end{flushleft}
\begin{remark}
  This is actually an improvement on the free Talagrand inequality. In fact, by taking $\mu=\sigma$, we immediately recover the free Talagrand inequality.
  Furthermore, by using the triangle inequality for $W_2$ with for centre, we get the semicircular law and again the free Talagrand inequality:
 \begin{equation}
     W_2(\mu,\nu)^2\leq 4H(\mu,\sigma)+4H(\nu,\sigma),\nonumber
 \end{equation}
  which, with our inequality, can be improved to a prefactor $2$, which can be interesting for applications. However the measure $\mu$ must be centered for the inequality to hold (and both compactly supported).
  \end{remark}
\begin{flushleft}
    Now, following the recent work of Tsuji \cite{noncenter}, we can extend our previous theorem without any barycenter restrictions, and thus obtain a more general free analogue of the sharp symmetrized Talagrand transportation cost inequality.
\end{flushleft}
\begin{theorem}\label{th9}
    Let $\mu$ and $\nu$ be compactly supported probability measures. Then, we have:
    \begin{equation}\label{TalagS2}
    W_2(\mu,\nu)^2\leq  2H(\mu,\sigma)+2H(\nu,\sigma)-2\bary(\mu)\cdot \bary(\nu),
 \end{equation}
 where $\bary(\mu):=\int xd\mu(x)$.
 \begin{flushleft}
 Equality occurs if and only if $\mu$ follows a semicircular distribution with variance $\sigma^2,\:\sigma >0$ and $\nu$ follows a semicircular distribution with variance $\sigma^{-2}$.
 \end{flushleft}
\end{theorem}
Before doing the proof, we need to introduce a preliminary lemma which yields to the the free counterpart of the Proposition $3.1$ in the paper \cite{noncenter} of Tsuji.
\begin{flushleft}
    For a probability measure $\mu$ on $\mathcal{P}(\mathbb{R}^d)$, we denote its barycenter 
as $\int xd\mu(x)$. For any $a\in \mathbb{R}^d$, we also denote the translation of $\mu$ by $a$ as the probability measure $\mu_a$ which is defined for all Borelian $A$ of $\mathbb{R}^d$ as $\mu_a(A):=\mu(A-a)$.
\end{flushleft}
\begin{lemma}
    Let $\nu$ be a compactly supported probability measure. Then for any other compactly supported measure $\mu$ and $a\in \mathbb{R}$, we have:
    \begin{equation}
        \chi_{\frac{x^2}{2}}(\mu_a)+\frac{1}{2}W_2(\mu_a,\nu)^2=\chi_{\frac{x^2}{2}}(\mu)+
    \frac{1}{2}W_2(\mu,\nu)^2-a\bary(\nu).
    \end{equation}
\end{lemma}
\begin{proof}
    First, note that since we are in the one-dimensional case, the non-commutativity disappears, so the Monge-Kantorovitch duality still holds, and we can still use the following equality, which was already proved in full generality by Tsuji in $\mathcal{P}_2(\mathbb{R}^d),\:d\geq 1$,
    \begin{equation}\label{19}
        \frac{1}{2}W_2(\mu_a,\nu)^2=
    \frac{1}{2}W_2(\mu,\nu)^2+a\bary(\mu)-a\bary(\nu)+\frac{a^2}{2}.
    \end{equation}
    Now it suffice to remark that we have,
    \begin{eqnarray}
        \chi_{\frac{x^2}{2}}(\mu_a)&=&-\frac{1}{2}\int x^2d\mu_a(x)+\iint \log\lvert t-s\rvert d\mu_a(s)d\mu_a(t)\nonumber\\
        &=&-\frac{1}{2}\int (x+a)^2d\mu+\iint \log\lvert t-s\rvert d\mu(s)d\mu(t)\nonumber\\
        &=&-\frac{1}{2}\int x^2d\mu+\iint \log\lvert t-s\rvert d\mu(s)d\mu(t)-a\int xd\mu(x)-\frac{a^2}{2}\nonumber\\
        &=&\chi_{\frac{x^2}{2}}(\mu)-a\bary(\mu)-\frac{a^2}{2}.
    \end{eqnarray}
    Plugging this into \eqref{19} concludes.
\end{proof} \qed
\begin{flushleft}
    We can now finish the proof of the sharp symmetrized free Talagrand TCI in the general case.
\end{flushleft}
\begin{proof}(Theorem~\ref{TalagS2})
    Its just an application of the previous lemma with $a:=\bary(\mu)$ to get the desired conclusion.
    Indeed, we get in this case,
    \begin{equation}
        \chi_{\frac{x^2}{2}}(\mu_a)+\frac{1}{2}W_2(\mu_a,\nu)^2=\chi_{\frac{x^2}{2}}(\mu)+
    \frac{1}{2}W_2(\mu,\nu)^2-\bary(\mu)\cdot \bary(\nu).
    \end{equation}
    which is equivalent to the conclusion.
\end{proof} \qed
\section{A free Blaschke-Santaló inequality}\label{sec3}
\begin{flushleft}
This section focuses on obtaining free analogues of some inequalities from convex geometry that have an intrinsic link to optimal transport theory, first discovered by Gozlan \cite{gozlan2} and recently formalised as a powerful duality by Fathi \cite{FathT}. 
\end{flushleft}
\begin{flushleft}
    Let us first recall that in dimension one (also true in the multivariate setting if we introduce a modified free entropy quantity defined as the Legendre transform of the multivariate free pressure), Hiai, Mizuo and Petz \cite{Hiai1, Hiai2} introduced the concept of free pressure and proved that the free entropy can be written as the (minus) Legendre transform of the free pressure. For convenience, we rewrite these formulas with the minus sign. 
\end{flushleft}

\begin{definition}
    Let $R>0$ and consider the free pressure defined as the Legendre transform of the the (minus) free entropy:
    \begin{equation}\label{press}
        \eta_R(h):=\sup\left\{\mu(h)+\chi(\mu), \mu \in \mathcal{M}([-R,R])\right\}, h\ \in C_{\mathbb{R}}([-R,R]), 
    \end{equation}
    where the dual pairing is given by
    $\mu(h):=\int hd\mu$, $h\in \mathcal{C}([-R,R]):=\mathcal{C}_{\mathbb{R}}([-R,R])$ (real-valued), and $\mu \in \mathcal{M}([-R,R])$,
\end{definition}
\begin{remark}\label{rm9}
A first important fact is that for $h\in \mathcal{C}([-R,R])$ (or considering $h\in \mathcal{C}(\mathbb{R})$ and denoting $h_{|[-R,R]}$ its restriction to $[-R,R]$), the quantity $\eta_R(h)$ is in fact equal to $\chi_h(\nu_h)$, i.e this quantitu is exactly the unique maximizer of \eqref{13} since we have the variational principle $\eta_R(h)=-\nu_h(h)+\chi(\nu_h)$, see for example Hiai, see page 710 in \cite{Hiai1}. We will therefore often state our results in terms of relative entropy, switching from one notation to the other at convenience (especially when using duality, we will need to introduce an appropriate cutoff to make sense of all the quantities needed).
\bigbreak
A second important fact is that the free pressure is in fact a convex decreasing function, which is also $1$ Lipschitz with respect to the uniform norm:
 \begin{equation}
     \lvert \eta_R(h_1)-\eta_R(h_2)\rvert\leq \lVert h_1-h_2\rVert_{\infty},\: h_1,h_2 \in \mathcal{C}_{\mathbb{R}}([-R,R]).
 \end{equation}
\end{remark}
\bigbreak
\begin{flushleft}
   From this definition, Hiai and a.l. \cite{Hiai1, Hiai2} proved that we can recover the free entropy as the (minus) Legendre transform of the free pressure, providing a free analogue to the well-known dual formulation of entropy as the Legendre transform of the log-Laplace functional in the classical case. 
\end{flushleft}

\begin{theorem}(Hiai, Mizuo, Petz; section 5 in \cite{Hiai1})\label{th12}
\newline
    For $\mu \in \mathcal{M}([-R,R])$, $h\in C_{\mathbb{R}}([-R,R])$, we have:
    \begin{enumerate}
        \item the free entropy is given as the Legendre transform of the free pressure:
    \begin{equation}\label{free-p}
        \chi(\mu)=\inf\left\{\mu(h)+\eta_{R}(h),\:h\in \mathcal{C}_{\mathbb{R}}([-R,R])\right\}.
    \end{equation}
\item
Moreover, we have:
\begin{equation}\label{freepress}
    \eta_R(h)=\underset{N\rightarrow \infty}{\lim}\bigg(\frac{1}{N^2}\log\bigg(\int_{(M_n^{sa})_{R}} \exp(-N^2\tr_N(h(M)))d\Lambda_N(M)\bigg)+\frac{1}{2}\log N\bigg),
\end{equation}
where $h(A)$ is defined by functional calculus and $(M_n^{sa})_{R}:=\left\{M\in M_n^{sa},\lVert M\rVert_{\infty}\leq R\right\}$
\end{enumerate}
\end{theorem}
The first result in this direction was a free analogue of the Brunn-Minkowski inequality obtained by Ledoux in \cite{Ledou} in the one-dimensional case by random matrix approximation, who later, in collaboration with Popescu, proved it by a mass transport argument. This inequality is actually in a Pr\'ekopa-Leindler form, and so far only exists in the one-dimensional case, which we will use to prove a free Blaschke-Santal\'o inequality. However, it is still not clear how to develop a multidimensional analogue of this free Brunn-Mikowski inequality, and it seems even more difficult to formulate a non-microstate extension which, despite the analytical flavour of the definitions, is much harder to manipulate in practice, especially when proving functional inequalities, since the matrix approximation no longer helps.
\begin{theorem}(Ledoux-Popescu \cite{Ledou,LP}, free Brunn-Minkowski inequality)\label{BML}
\newline
    Let $U_1,U_2,U_2$ be real-valued continuous functions on $\mathbb{R}$ satisfying the growth assumptions \eqref{gibbs} and such that for some $\theta\in (0,1)$, for all $x,y\in \mathbb{R}$,
    \begin{equation}
        U_3(\theta x+(1-\theta)y)\leq \theta U_1(x)+(1-\theta)U_2(y).\nonumber
    \end{equation}
    Then \begin{equation}
        \chi_{U_3}(\nu_{U_3})\geq \theta\chi_{U_1}(\nu_{U_1})+(1-\theta)\chi_{U_2}(\nu_{U_2}).\end{equation}
        
With the notations involving a cutoff, if $R>0$ is such that $\supp(\nu_{U_1}),\supp(\nu_{U_2}),\supp(\nu_{U_3})\subset[-R,R]$, we have:
\begin{equation}
    \eta_R(U_3)\geq \theta \eta_R(U_1)+(1-\theta)\eta_R(U_2).
\end{equation}
\end{theorem}
The following inequality is known as the functional Blaschke-Santaló inequality since it is the functional generalization of Santaló's result for convex bodies (note that there are many different variants, and the one we use in the sequel is due to Lehec \cite{Lehec}, who generalized the first result for even functions proved by Ball \cite{Ball}). 
\newline
It was shown by Fathi \cite{FathT} that this inequality is the dual formulation of the Sharp's symmetrized Talagrand inequality for the Gaussian measure.
This classical Santaló inequality is in fact the starting point of our investigations to provide its free version using random matrix approximations.
\begin{theorem}(Functional Santaló inequality)
Let $f$ and $g$ be mesurable functions on $\mathbb{R}^d$ such that we have $f(x)+g(y)\leq -x\cdot y,\: \forall x,y \in \mathbb{R}^d$, if $e^f$ (or $e^g$) have barycenter zero, then:
\begin{equation}
    \bigg(\int e^fdx\bigg)\bigg(\int e^gdx\bigg)\leq (2\pi)^d,
\end{equation}
equality occurs if and if only there exists $C\in \mathbb{R}$ and and a positive definite matrix $A$, such that 
\begin{equation}
    f(x)=C-Ax\cdot x,\: \mbox{and}\: g(y)=-C-A^{-1}y\cdot y
\end{equation}
$\mbox{a.e.}$ in $\mathbb{R}^n$.
\end{theorem}
This theorem is in fact the functional version of the famous Blaschke-Santaló inequality for convex bodies, first proved in dimension $n\leq 3$ by Blaschke \cite{blaschke} and in full generality by Santal\'o \cite{santa} using Steiner symmetrization and the Brunn-Minkowski inequality.
\begin{theorem}(Blaschke-Santal\'o inequality)
    Let $K\subset\mathbb{R}^d$ a centrally symmetric ($K=-K$) convex body and $K^{\circ}$ its polar set, then the Mähler-Volume of $K$ is bounded as follows:
    \begin{equation}\label{bs}
        M(K):=\Vol_d(K)\Vol_d(K^{\circ})\leq (\Vol_d(B_2^d))^2
    \end{equation} 
    Equality occurs in \eqref{bs} if and if only $K$ is an ellipsoid and $B_2^d$ denotes the Euclidean Ball w.r.t the standard inner product on $\mathbb{R}^d$.
\end{theorem}
\begin{flushleft}
The next step is to give an indication of the free version of the functional Blaschke-Santal\'o inequality. This will be suggested by large matrix approximations and the classical version of the Blaschke-Santal\'o inequality as follows.
\end{flushleft}
\begin{flushleft}
    Indeed, let's $f,g$ be continuous functions satisfying the inequality $f(x)+g(y)\geq xy$ for all $x,y\in \mathbb{R}$.
    \end{flushleft}
\begin{flushleft}
    Let's denote for convenience the "\textit{micro-pressure}":
    \begin{equation}
        \eta_{R,N}(f)=\log\bigg(\int_{(M_n^{sa})_{R}} \exp(-N^2\tr_N(f(M)))d\Lambda_N(M)\bigg),\: \forall n\geq 1.\nonumber
    \end{equation}
    We choose $R>0$ as a suitably chosen cutoff.
    We also assume, without loss of generality, that the following probability measure, called the \textit{micro Gibbs measure} $\mu_{R}^f$ on the set of Hermitian matrices:
    \begin{equation}
        d\mu^f_{R,N}(M)=\frac{1}{Z_{R,N}^f}\exp(-N^2\tr_N(f(M)))1_{\left\{\lVert M\rVert_{\infty}\leq R\right\}}\:d\Lambda_N(M),\nonumber
    \end{equation}
    \newline
    And its truncated probability measure denoted $\mu_{R,N}^f$ on the set of Hermitian matrices (with the proper renormalization constant $Z_{R,N}^f$):
    \begin{equation}
        d\mu^f_{R,N}(M)=\frac{1}{Z_{R,N}^f}\exp(-N^2\tr_N(f(M)))1_{\left\{\lVert M\rVert_{\infty}\leq R\right\}}\:d\Lambda_N(M),\nonumber
    \end{equation}
    We know assume that $\mu_{R,N}^f$ has barycenter zero for (at least) $N$ large, i.e.
    \begin{equation}
         \int\tr_N(\Phi(M))d\mu_{N}^f(M)=0
    \end{equation}
    This is easily satisfied, for example, if the potential $f$ is even (separating $M_N^{s.a}$ into the positive and negative cones), and is already an interesting case for applications, see e.g. Kolesnikov, Werner \cite{Koles} for many new results on generalised Blaschke-Santal\'o inequalities in the classical case.
    \end{flushleft}
   Now since we have the following weak convergence, i.e. for any bounded continuous function $\Phi:\mathbb{R}\rightarrow \mathbb{R}$,
   \begin{equation}
       \int \tr_N(\Phi(M))d\mu_{N}^f(M)\rightarrow \int \Phi(x)d\mu_f(x)
   \end{equation}
   In the case where the convergence of the moments is ensured (i.e. this convergence is extended to $\Phi(x)=x^k, k\geq 0$), e.g. see \cite{jekel} $f$, we directly find the barycenter condition at the limit for the free Gibbs measure.
\begin{equation}
        \int xd\nu_f(x)=0.\nonumber
    \end{equation}
   If we choose $R>0$ large enough, we still have the weak convergence for the truncated probability measure
   \begin{equation}
       \int \tr_N(\Phi(M))d\mu_{R,N}^f(M)\rightarrow \int \Phi(x)d\nu_f(x)
   \end{equation}
\begin{flushleft}
   Since $\mu_{R,N}^f$ is compactly supported as well as the free Gibbs measure $\nu_f$ (both have a fortiori moments of any order), hence, by choosing $\Phi: \mathbb{R}\rightarrow \mathbb{R}$ as a $\mathcal{C}_c^{\infty}$ function such that $\Phi(t)=t,\: t\leq R$ with cutoff $R>0$ large enough (take $R>\frac{\diam(\supp(\nu_f))}{2}$, which is sufficient), we still get
    \begin{equation}
        \int xd\nu_f(x)=0.\nonumber
    \end{equation} 
    \end{flushleft}
    So, to apply the classical Blaschke-Santal\'o inequality to $M_N^{s.a}$, the last step is to make sure that the inequality $f(x)+g(y)\geq xy$ is preserved for self-adjoint matrices. 
    Fortunately, this is shown in the next proposition.
    \newline
    \begin{proposition}\label{prop2}
        Let $f,g:\mathbb{R}\rightarrow \mathbb{R}$ be such that $\forall (x,y)\in \mathbb{R}^2\:f(x)+g(y)\geq xy$. Let $x,y\in M_N^{s.a}$ two self-adjoint $N\times N$ matrices. Then we have $N^2\tr_N(f(x))+N^2\tr_N(g(y))\geq N^2\tr_N(xy)$.
    \end{proposition}
 
    \begin{proof}
         To prove this, we consider for each $n\geq 1$ and for each $x,y\in M_N^{sa}$,
    \begin{equation}
        F(x):=N^2\tr_N(f(x))\nonumber,
    \end{equation}
    and 
    \begin{equation}
        G(y):=N^2\tr_N(g(y))\nonumber,
    \end{equation}
  \begin{flushleft}
    Now, since $x,y$ are self-adjoint, there are two (complex) orthonormal bases where $x,y$ are similar to $D=diag(\alpha_1,\ldots,\alpha_N)$ and $E=diag(\beta_1,\ldots,\beta_N)$, respectively, from which we deduce that
    \begin{equation}
        F(x)=N^2\tr_N(f(D))=N\sum_{i=1}^Nf(\alpha_i),\: \mbox{and}\:  G(y)=N^2\tr_N(f(E))=N\sum_{i=1}^Ng(\beta_i)\nonumber,
    \end{equation}
    with $\forall 1\leq i\leq N$, $\lambda_i:=f(\alpha_i) \:\mbox{and}\:\mu_i:=g(\beta_i)\in \mathbb{R}$.
    \bigbreak
    Now, without loss of generality, we can assume that
    \begin{equation}
        \lambda_1\leq \ldots \leq \lambda_N\:\mbox{and}\:\mu_1\leq \ldots\leq \mu_N,\nonumber
    \end{equation}
   Indeed, both $F(x)$ and $G(y)$ are defined via a sum over the respective eingavlue of $x,y$, thus are invariant by permutations of the respective eigenvalues of $x$ and $y$, we can choose a permutation of $\lambda_i$ or $\mu_i$ that respects the increasing order.
   So we can easily deduce from the assumption $f(x)+g(y)\geq xy$ that
    \begin{equation}
        F(x)+G(y)\geq N\sum_{i=1}^N \lambda_i\mu_i\nonumber
    \end{equation}
    Now by using the Hoeffman-Wielandt inequality (see e.g. \cite{HW}), we get
    \begin{equation}
         \sum_{i=1}^N\lambda_i\mu_i\geq \Tr(xy),\nonumber
    \end{equation}
    which implies finally the conclusion,
    \begin{equation}
        F(x)+G(y)\geq N^2\tr_N(xy).\nonumber
    \end{equation}
    \end{flushleft}
     \end{proof} 
\qed
     \begin{flushleft}
We now introduce the two following quantities which corresponds respectively to the exponential of the "micro-pressure":
    \begin{equation}
        \forall N\geq1,\: \tilde{\eta}_{R,N}(f):=\int_{(M_n^{sa})_{R}}\exp(-N^2\tr_N(f(M)))d\Lambda_N(M),\: f\in \mathcal{C}_{\mathbb{R}}(\mathbb{R}),\nonumber
    \end{equation}
    and its renormalized version
    \begin{equation}
        {\eta}^1_{R,N}(f):=N^{\frac{N^2}{2}}\int_{(M_n^{sa})_{R}}\exp(-N^2\tr_N(f(M)))d\Lambda_N(M)=N^{\frac{N^2}{2}}\tilde{\eta}_{R,N}(f).\nonumber
    \end{equation}
    \end{flushleft}
    \begin{flushleft}
        Note here, that $N^2\tr_N(M^2):=N\Tr_N(M^2)=N\lVert M\rVert_{HS}^2$, and that the Hilbert-Schmidt norm $\lVert\cdot \rVert_{HS}$ corresponds to the Euclidean norm on $\mathbb{R}^{N^2}$ under the isometry
        \begin{equation}\label{isom}
            M = [M_{ij}] \in M_N^{s.a} \mapsto ((M_{ii})_{1\leq i\leq n} , (\sqrt{2}\Re(M_{ij}))_{i<j} , (\sqrt{2} \Im(M_{ij}))_{i<j})\in \mathbb{R}^{N^2}.
        \end{equation}
    \end{flushleft}
    \begin{flushleft}
     So if the measure $\mu_f^{(n)}$ has a barycenter of zero (so its limiting equilibrium measure also has a barycenter of zero), using the functional Santaló inequality with $d=N^2$ we get $\forall N\geq1$. 
    \begin{equation}
        \log({\eta}^1_{R,N}(f))+\log({\eta}^1_{R,N}(g))\leq N^2\log(2\pi).
    \end{equation}
    which is equivalent to
    \begin{equation}
        \frac{1}{N^2}log(\tilde{\eta}_{R,N}(f))+\frac{1}{N^2}\log(\tilde{\eta}_{R,N}(g))\leq \log(2\pi)-\log(N),\nonumber
    \end{equation}
    and finally this (adding on both sides $\log N$),
    \begin{equation}\label{micro}
        \bigg(\frac{1}{N^2}log(\tilde{\eta}_{R,N}(f))+\frac{1}{2}\log N\bigg)+\bigg(\frac{1}{N^2}\log(\tilde{\eta}_{R,N}(g))+\frac{1}{2}\log N\bigg)\leq \log(2\pi).
    \end{equation}
    Taking the limit in \eqref{micro}, since both terms of the left side are well defined, we get the free functional Santal\'o inequality
    \begin{equation}\label{santlo}
        \eta_R(f)+\eta_R(g)\leq \log(2\pi),
    \end{equation}
    The conclusion follows by taking the exponential.
    \end{flushleft}
    \begin{remark}
        We recall that the density of $GUE$ with respect to the Lebesgue measure is given for $N\geq 1$ by
    \begin{equation}
        d\sigma_N(M)=2^{-\frac{N}{2}}\pi^{-\frac{N^2}{2}}\exp\bigg(-\frac{tr(M^2)}{2}\bigg)d\Lambda_N(M).
    \end{equation}
    Using the isometry given by \eqref{isom} (note in particular the importance of the coefficient $\sqrt{2}$ to renormalize correctly) we can explain why we get the "right" rate $N^2$ at the end of the calculations, which is necessary to use the large deviation techniques. 
    \end{remark}
This reasoning (which is perfectly formal in the even case) leading to the first inequality \eqref{santlo} then suggests the following free form of the functional Blaschke-Santal\'o inequality.
\begin{flushleft}
    In fact, in one of Lehec's proof \cite{Lehec}, his arguments are based on an induction over the dimension and on a one-dimensional logarithmic form of the Prekopa-Leindler inequality that we recall here (the cases of equality can also be deduced as for the classical Prekopá-Leindler inequality, which were characterized by Dubuc \cite{dubuc}). In another of his work, Lehec used the Yao-Yao partition \cite{Lehec3}) to obtain another new proof.
    \end{flushleft}
    \begin{lemma}(Logarithmic Prekopa-Leindler)
    \newline
        Let $\phi_1,\phi_2$ be non-negative Borel functions on $\mathbb{R}_+$. If $\phi_1(t)\phi_2(t)\leq e^{-st}$ for every $s,t\in \mathbb{R}_+$, then
        \begin{equation}\label{prek}
            \int_{\mathbb{R}_+} \phi_1(s)ds\int_{\mathbb{R}_+}\phi_2(s)ds\leq \frac{\pi}{2}.
        \end{equation}
    \end{lemma}
More generally, when we strenghned a bit the symmetry assumption to an unconditional version involving geometric means
 \begin{lemma}(Prekopa-Leindler for the geometric mean)
    \newline
        Let $\phi_1,\phi_2:\mathbb{R}^n\rightarrow\mathbb{R}$ be Borel measurable unconditional functions such that
        \newline
        $\phi_1(x_1,\ldots,x_n)+\phi_2(y_1,\ldots,y_n)\leq -\langle x,y\rangle$ for every $(x_1,\ldots,x_n)\:\mbox{and}\: (y_1,\ldots,y_n)\in\mathbb{R}_+^n$, then
        \begin{equation}\label{prek2}
            \int_{\mathbb{R}^n_+} e^{\phi_1}\int_{\mathbb{R}^n_+}e^{\phi_2}\leq \bigg(\frac{\pi}{2}\bigg)^n.
        \end{equation}
    \end{lemma}
As for the Santa\'o inequality, which states that Gaussian densities are the unique maximizers of the Blaschke-Santal\'o functional, this lemma can be reinterpreted by saying that \textit{half-gaussian} densities maximize the ``positive`` Blaschke-Santal\'o functional.
\bigbreak
It turns out that this statement can also be translated in the free case as a logarithmic version of the free Brunn-Minkowski inequality, and tells us that the maximum of the midpoint of the free pressures is obtained for the quarter-circular law. 
We also recall that the usual self-adjoint free entropy might not be necessarily the most adapted quantity to deal with measure on $\mathbb{R}_+$ or more generally inside the positive cone of a von Neumann algebra. Hiai and Petz in \cite{Hia4} introduced a free entropy quantity adapted to deal with polar decompositions, and which is well adapter to compute free entropy for elements in the positive cone of a von Neumann algebra $\mathcal{M}$. This free entropy adapted to the polar decomposition is in fact defined for $h\in \mathcal{M}_+$ with distribution $\mu_h\in \mathcal{M}(\mathbb{R}_+)$ as,
\begin{eqnarray}
    \chi^+(\mu_h)&:=&\chi(\mu_h)+\frac{1}{2}\log\bigg(\frac{\pi}{2}\bigg)+\frac{3}{4}\\\nonumber
    &=& L(\mu_h)+\log \pi+\frac{3}{2}
\end{eqnarray}
Hiai and Petz also proved (see remark $2.4$ in \cite{Hia4} or Proposition $3.2$ in \cite{HiaP}) that the unique maximizer (with maximal value $\log(\pi eC)$) of the functional $\mathcal{M}_+\ni h \mapsto \chi^+(\mu_h)$ under the constraint $\tau(h)\leq C$ is in fact achieved when $h$ as an analytic distribution given by
\begin{equation}
    \frac{\sqrt{4Ct-t^2}}{2\pi Ct}1_{[0,4C]}(t)dt.
\end{equation}
This means nothing but to say that $h^{\frac{1}{2}}$ is the quarter-circular law of radii $2\sqrt{C}$. When $C=1$, $h^{\frac{1}{2}}$ follows the standard quarter-circular law which we recall is the pushforward of a Marcenko-Pastur law under the map $x\mapsto \sqrt{x}$, or the distribution of the polar part of a semicircular operator (see Voiculescu \cite{Voic7} or Banica \cite{Banica}). Moreover, it is well known (see Ledoux, Popescu \cite{LP}) that this standard Marcenko-Pastur law is the unique maximizer of the relative entropy functional with potential $\Id(x):=x$ over $\mathcal{P}([0,\infty])$, i.e. it maximizes the relative entropy $\chi_{\Id}(\cdot)$ over the space of measures supported on the positive real line.
\begin{flushleft}

However, we need to introduce a modified quantity adapted to deal with probability measures supported on the positive real line and associated to a potential on this positive real line. The problems are that the constants are mostly not taken into account (because they vanish in most of the free functional inequalities) of all the previous related work already evoked.
\begin{definition}
    Let $\mu\in \mathcal{M}(\mathbb{R}_+)$ and a potential $V:\mathbb{R}_+\rightarrow \overline{\mathbb{R}}$ continuous satisfying condition \eqref{gibbs}. Then, we set:
    \begin{equation}
        \chi^+_{V}(\mu):=\chi^+(\mu)-\int V(x)d\mu(x).
    \end{equation}
\end{definition}
This entropy appears also as the large deviation rate of self-adjoint random matrices (see Hiai and Petz \cite{HiaP}) conditionned to stay in the positive cone $M_n^{s.a,+}:=\left\{M\in M_n^{s.a},\: M\geq 0\right\}$ with thus all positive eigenvalues.
\end{flushleft}
Using now Ledoux and Popescu's notations (see equations $(10.2)\:\mbox{and}\: (10.3)$ in \cite{LP}), we have the following slight modification (with the new constant $\frac{1}{2}\log 2$) of their proposition that will be use in the next proof. 
\begin{proposition}\label{prop18}
      For a measure $\mu$ on $[0,\infty)$ we defined its associated symmetric measure $\tilde{\mu}$ by setting for every Borel subset of $[0,\infty)$ \begin{equation}
        \mu(F):=\tilde{\mu}(\left\{x, x^2\in F\right\})
    \end{equation}
     We then have that the minimizer of $\chi_{\tilde{V}}(\cdot)$ is maximized for $\nu_{\tilde{V}}=\tilde{\nu}_V$ where for a potential $V:[0,\infty)\rightarrow \overline{\mathbb{R}}:=\mathbb{R}\cup\left\{+\infty\right\}$ we set $\tilde{V}(x):=\frac{V(x^2)}{2}$, and moreover we have $\chi^+_{V}(\nu_V)=2(\chi_{\tilde{V}}(\nu_{\tilde{V}})-\frac{1}{2}\log 2)$.
\end{proposition}

\begin{flushleft}
    The following lemma is the logarithmic form of the free Prekopa-Leindler inequality. We omit the proof as is a standard to obtain from random matrix approximations as for the proof of \eqref{BML}.
\end{flushleft}

\begin{lemma}\label{lemma18}
    Let $U_1,U_2,U_2$ be real-valued continuous functions on $[0,+\infty)$ satisfying the growth assumptions \eqref{gibbs}. Suppose that for all $x,y\in \mathbb{R}$,
    \begin{equation}
        xy\leq \frac{U_1(x^2)}{2}+\frac{U_2(y^2)}{2}.\nonumber
    \end{equation}
    Then,
    \begin{equation}
         \frac{1}{2}\chi_{U_1}^+(\nu_{U_1})+\frac{1}{2}\chi^+_{U_2}(\nu_{U_2})\leq \log \pi.\end{equation}
\end{lemma}
\begin{flushleft}
Suppose for now that we would have the stronger hypothesis that $\forall x,y \in \mathbb{R},\:f(x)+g(y)\geq \frac{(x+y)^2}{2}$, then applying the free Brunn Minkowski's inequality \eqref{BML}, the free Blaschke-Santaló inequality would follow easily by setting $U_1:=f,\:U_2:=g,\:\mbox{and}\:U_3:=x^2$ with $\theta:=\frac{1}{2}$.
    The idea of the following is that there is nothing in between this two lower bounds $(x,y)\mapsto xy$ and $(x,y)\mapsto \frac{(x+y)^2}{2}$ in the centered case.
\end{flushleft}

\begin{theorem}(Free functional Blaschke-Santaló inequality)\label{san}.
\newline
Let $f,g$ be continuous functions, such that $f(x)+g(y)\geq xy$ for all $x,y \in \mathbb{R}$. If the equilibrium measure $\nu_{f}$ (or $\nu_{g}$) solution of the variational problem \eqref{variation} has barycenter zero, then 

\begin{equation}
    e^{\chi_f(\nu_f)+\chi_g(\nu_g)}\leq 2\pi.
\end{equation}
Or with the free pressure,
for $R\geq2$, such that $\supp(\nu_f),\supp(\nu_g)\subset[-R,R]$, we have
\begin{equation}\label{62}
    e^{\eta_R(f)+\eta_R(g)}\leq 2\pi,
\end{equation}
\end{theorem}
More generally, we have the following more general statement which gives the existence of a Santal\'o point (i.e. an appropriate shift) when none of the measure is supposed to have barycenter zero.
\begin{corollary}
    (Non-centered free Blaschke-Santaló inequality)\label{san2}.
\newline
Let $f$ be a continuous functions. Then, there exists $z\in \mathbb{R}$ such that for every other continuous function $g$ 
\begin{equation}\label{santal}
    (\forall\:x,y \in \mathbb{R},\:f(z+x)+g(y)\geq xy) \implies  e^{\chi_{f}(\nu_{f})+\chi_g(\nu_g)}\leq 2\pi.
\end{equation}
\end{corollary}
\begin{proof}(Theorem \ref{san} and corollary \ref{san2})
\begin{flushleft}

For simplicity, we choose $R\geq 2$ large enough that both $\nu_f$ and $\nu_g$, as well as the standard semicircular distribution, have their support in $[-R,R]$. 

\end{flushleft}
We will explain how the how the existence of such a $z\in \mathbb{R}$ (drawing similiarity with the even case), then explain how to deduce that we can choose $z=0$ when we only assume that $\nu_g$ is centered (with no symmetry assumptions). Finally, we will look at the cases of equality right after that by using a complete different tensorization technique.

\begin{flushleft}
   \textbf{Step 1: The general case}
    \bigbreak
    Now let $z\in \mathbb{R}$ be such that (the existence in the general case of such a point $z$ is easy to deduce, since it is just a "log"-median        of the positive free relative entropy $\chi^+_{\phi^+}(\nu_{\phi^+})$, and the reader might for simplicity think of it as $z=0$ as in the even case),
\begin{equation}
     \chi_{\phi^+}(\nu_{\phi^+})=2(\chi_{f_z}(\nu_{f_z})-\frac{1}{2}\log2).
     \end{equation}
     where we again denote $f_z$ as the function $f_z(\cdot):=f( z+\cdot)$ as the usual shift by $z$.
\bigbreak
Indeed, from proposition \ref{prop18}, we know that
     \begin{equation}
         \chi_{\phi^+}^+(\nu_{\phi^+})=2(\chi_{f{\circ \lvert \cdot\rvert}}-\frac{1}{2}\log 2),
     \end{equation}
     where $f\circ\lvert \cdot \rvert$ denotes the composition of $f$ with the absolute value.
     So in the even case, we would get
     \begin{equation}
         \chi_{\phi^+}(\nu_{\phi^+})=2(\chi_{f}-\frac{1}{2}\log 2).\end{equation}

In this general case, we would have to choose a suitable replacement of $g$ to ensure that the condition \eqref{santal} is satisfied, i.e. we replace $g$ by $\bar{g}_z(y):=g(y)-zy$, since
 \begin{equation}
    f_z(x)+\bar{g}_z(y)-zy:=f(z+x)+g(y)-zy\geq(z+x)y-zy=xy.
\end{equation} 
For such a choice we use the same conventions as for $\bar{g}_z$, with the notation $\psi_z$ instead. 
\bigbreak

 We can apply the lemma \ref{lemma18} twice, first to $\mathbb{R}_+$ and to the restriction $\phi^+$ and $\bar{g}_z^+$ to $\mathbb{R}_+$, to get
    \begin{eqnarray}
        e^{\frac{1}{2}\chi_{\phi^+}(\nu_{\phi^+})+\frac{1}{2}\chi_{\bar{g}_z^+}(\nu_{\bar{g}_z^+})}\leq \pi.
    \end{eqnarray} 
    And thus
    \begin{equation}
        e^{\chi_{f_z}(\nu_{f_z})+\frac{1}{2}\chi^+_{\bar{g}_z^+}(\nu_{\bar{g}_z^+})}\leq \sqrt{2}\pi
    \end{equation}
    But since the free relative entropy (or free pressure) is invariant by translation, i.e. $\chi_{f_z}(\nu_{f_z})=\chi_f(\nu_{f})$ (again, we refer to \eqref{lebesgue} or e.g. \eqref{freepress} to see this invariance as the level of random matrices). 
    \bigbreak
    Applying the same reasoning on $\mathbb{R}_-$ and summing the two inequalities and we get the desired conclusion.
\end{flushleft}
\bigbreak
\begin{flushleft}
\textbf{Step 2: The centered case}
\bigbreak
Now we relax the even hypothesis a bit. In fact, in this part we prove that we can actually take $z=0$ if $\nu_g$ is centered. 
\newline
We consider as before the function $\bar{g}_z(y):=g(y)-zy$, which together with $f_z$ satisfies the quadratic growth condition \eqref{santal}, since
 \begin{equation}
    f_z(x)+\bar{g}_z(y)-zy:=f(z+x)+g(y)-zy\geq(z+x)y-zy=xy.
\end{equation} 
Now we have by hypothesis, the existence $z\in \mathbb{R}$ such that
\begin{equation}\label{non}
        e^{\chi_{f}(\nu_{f})+\chi_{\bar{g}_z}(\nu_{\bar{g}_z})}\leq 2\pi.
    \end{equation}
    and the goal is to show that we can choose in fact $z=0$.
    \bigbreak
    Since we have assumed that $g$ has a barycenter zero, by definition of $\nu_{\bar{g}_z}$, which is the maximiser of the free relative entropy functional associated with the potential $\bar{g}_z$, now for any probability measure $\mu\in \mathcal{P}(\mathbb{R})$ we have
    \begin{eqnarray}
        \chi_{g}(\mu)&:=&\iint \log\lvert x-y\rvert d\mu(x)d\mu(y)-\int gd\mu(y)\nonumber\\
        &=& \iint \log\lvert x-y\rvert d\mu(x)d\mu(y)-\int [g(y)-zy]d\mu(y)-z\int yd\mu(y)\nonumber\\
    &=& \underbrace{\iint \log\lvert x-y\rvert d\mu(x)d\mu(y)-\int \bar{g}_z(y)d\mu(y)}_{\leq \chi_{\bar{g}_z}(\nu_{\bar{g}_z})}-z\int yd\mu(y)\nonumber\\
    &\leq& \chi_{\bar{g}_z}(\nu_{\bar{g}_z})-z\int yd\mu(y)
    \end{eqnarray}
    Applying this inequality to $\mu=\nu_{g}$ which has barycenter zero and we get that
    \begin{equation}
        \chi_{g}(\nu_g) \leq \chi_{\bar{g}_y}(\nu_{\bar{g}_y})-z\underbrace{\int yd\nu_g(y)}_{=0},
    \end{equation}
     which with combine with the previous inequality \eqref{non} tells us that we can therefore choose $z=0$ in the centered case.
\end{flushleft}
\bigbreak
\begin{flushleft}
\textbf{Step 3: The cases of equality}
\bigbreak
To understand the cases of equality, it's better to work in terms of free pressure (the desired properties appear easier with the cutoff). So let's assume, without loss of generality, that $\int x^2d\nu_f(x)=1$. To reach the conclusion, we will use a tensorization procedure that exploits the fact that the free entropy of free microstates is additive for free tuples of non-commutative random variables. Regarding the free pressure, Hiai reformulated the proposition we need: in dimension $d=2$ (also true in any dimension), the free pressure is additive when restricted to separable real-valued continuous functions, i.e. for $d=1$, the free pressure is additive, e.g. for $h(x,y):=f(x)+g(y)$ we have, using the point $(6)$, proposition $2.3$ in \cite{Hiai2}, for continuous functions $f,g:\mathbb{R}\rightarrow \mathbb{R}$,
\begin{eqnarray}
\eta_R(f)+\eta_R(g)&=&\eta_R^2(h),\:\mbox{where}\: h(x,y)=f(x)+g(y),
    \end{eqnarray}
    \bigbreak
    where more generally for $h\in (\mathcal{A}_{(R)}^2)^{s.a}$, the self-adjoint part of $\mathcal{A}_{(R)}^2:=\mathcal{C}([-R,R])^{\star 2}$ which is the $2$-fold universal free product $C^*$-algebra of $\mathcal{C}([-R,R])$, we denote:
    \begin{equation}
        \eta_R^2(h):=\underset{N\rightarrow \infty}{\lim}\bigg(\frac{1}{N^2}\log\bigg(\int_{(M_n^{sa})^2_{R}} \exp(-N^2\tr_N(h(M_1,M_2)))d\Lambda^{\otimes 2}_N(M_1,M_2)\bigg)+\log N\bigg).
    \end{equation}
    \end{flushleft}
\begin{flushleft}
Recall that $\chi({\sigma^{* 2}})=\log(2\pi e)$ (which is the multi-dimensional free Gibbs laws corresponding to the potential $V_1=\frac{x^2}{2}+\frac{y^2}{2}$) and that the free entropy and the free-entropy quantity introduce by Hiai in \cite{Hiai2} as the Legendre transform of the free pressure coincide on free uniformly $R$-bounded tuples (see theorem $4.5$ $(iii)$ in \cite{Hiai2}). Hence, from the equality in \eqref{62}, we deduce
\begin{equation}
\log(2\pi)=\eta_R^2(h)=\eta_R^2\bigg(\frac{x^2}{2}+\frac{y^2}{2}\bigg)
\end{equation}
\end{flushleft}
\begin{flushleft}
Again, by using that the free pressure $\eta_R^2$ is decreasing (in the sense of operator) as well as the convexity of the free pressure, and that for every $R\geq 2$, $V_1\in (\mathcal{A}_{(R)}^2)^{s.a}$, we must have
\begin{equation}
f(x)+g(y)=\frac{x^2}{2}+\frac{y^2}{2},
\end{equation}
since the free entropy under the variance condition is maximized (in dimension $n=2$ for the sum of the variances is equals to two) for the standard semicircular system, we deduce that $f(x)=g(x)=\frac{x^2}{2}$ and therefore that the two free Gibbs measures $\nu_f$ and $\nu_g$ must follow a standard semicircular distribution.
\end{flushleft}
\begin{flushleft}
The general case follows by assuming that $\int x^2d\nu_f(x)=\rho^{-1},\:\rho>0$, from which we deduce that again the two measures are semicircular, with respective variances $\rho^{-1}$ and $\rho$.
\end{flushleft}
\end{proof} \qed

\begin{remark}
    The above random matrix heuristics can be formalised if we have convergence of the moments of the random matrix models. This is ensured, for example, if the potential $V$ is $c$-uniformly convex and semi-concave, see the recent important work of Jekel \cite{jekel,jekel2} who uses PDE techniques to prove that the log-concave measure on $(M_N^{s. a})^n$, given by uniformly convex and semi-concave (not necessarily constant) potentials approximable by trace polynomials, satisfies the concentration of the measure around a non-commutative law in the large $N$ limit, and uses this to provide new deep results on the equality cases between the non-microstates and microstates free entropy. \end{remark}

For convenience, by taking the supremum over the cutoff $R>0$ (or choosing this cutoff $R$ large enough to ensure that the limiting equilibrium measure has support in $[-R,R]$) in the definition \ref{press}, we can obtain a representation of the free entropy for any compactly supported probability measure without specifying the cutoff exactly. This is a saturation property and has been proved in various contexts, see for example Ueda, Theorem $4.5$ in \cite{Hiai2} or Biane and Dabrowski end of page $674$) of \cite{BiaD} for this statement. 
\begin{remark}
 In fact, this amounts to working with the space of compactly supported probability measures, denoted again as $\mathcal{M}_c:=\bigcup_{R>0} \mathcal{M}([-R,R])$. Equivalently, for a general continuous $f:\mathbb{R}\rightarrow \mathbb{R}$ satisfying assumptions \eqref{gibbs}, we know that the relative free entropy functional
\begin{equation}
\mu\ni \mathcal{P}(\mathbb{R})\mapsto -\mu(f)+\chi(\mu),
\end{equation}
admits a unique (compactly supported) maximiser $\nu_f$.
\begin{flushleft}
We let $R_f=\inf\left\{R>0, \supp(\nu_f)\subset[-R,R]\right\}$ and we set (we trivially extend the definition for measure supported on subset by setting zero mass outside its support).
\begin{equation}
    \eta(f):=\eta_{R_f\vee 2}(f)=\underset{R\geq R_f\vee 2}{\sup} \eta_{R_f}(f).\end{equation}
    The cutoff $R\geq 2$ is actually chosen to include the appropriate standard semicircular potential.
\end{flushleft}
\end{remark}

\begin{flushleft}
Unfortunately, we cannot yet reach the duality noted by Fathi in the classical case (see Proposition 1.3 in \cite{FathT}): the sharp symmetrized Talagrand inequality (abbreviated SSFTI) is dual to the functional Santal\'o inequality. The main difficulty in the free case is to prove the inverse implication (free Santal\'o implying SSFTI). This is explained in detail in the remark \ref{rmk9}. 
\end{flushleft}
The following proof is largely inspired by Fathi's techniques in the classical case. The crucial fact here is that the dimension is one. This is essential for the application of the Monge-Kantorovitch duality.
\bigbreak
We also recall that non-commutative Monge-Kantorovitch duality has recently been explored in an important paper by Jekel and a.l. \cite{JS2}. This requires the introduction of a new class of convex functions, and there may also be subtleties about optimal coupling, as non-commutative laws living in a finite dimensional algebra may require an infinite dimensional algebra to couple optimally. Moreover, the optimality of transport maps (as the one considered by Guionnet and Shlyakhtenko) is only proved for some cases, see remark $9.9$ in \cite{jek3}.
Furthermore, this duality is currently only available for quadratic costs, and has not yet been established for arbitrary transport costs.
\begin{proposition}\label{prop8}
    The sharp symmetrized free Talagrand inequality implies the free functional Santaló inequality.
\end{proposition}
\begin{proof}
 We reformulate this time (SSFTI) with the maximal correlation functional and the free entropy $\chi$, i.e (SSFTI) is equivalent to:
    \begin{equation}\label{ssfti}
        \underset{\pi}\inf -\int x\cdot y\:d\pi \leq -\chi(\mu)-\chi(\nu)+\log(2\pi),
    \end{equation}
    or equivalently
        \begin{equation}\label{ssfti2}
        \underset{\pi}\sup \int x\cdot y\:d\pi \geq \chi(\mu)+\chi(\nu)-\log(2\pi),
    \end{equation}
where $\pi$ denotes the set of transport plans with marginals $\mu$ and $\nu$.
\begin{flushleft}

Since we are in dimension one (hence the non-commutativity vanishes), the left side is still a transport cost with the (still nice enough) cost function $x\cdot y$ instead of $\lvert x-y\rvert^2$.
\bigbreak
Hence by Kantorovitch duality (see special case 5.1.16 in \cite{V}):
\end{flushleft}
\begin{eqnarray}\label{kanto}
    \underset{\pi}\sup \int x\cdot y\:d\pi&=&\underset{f(x)+g(y)\geq xy}{\inf} \int fd\mu +\int gd\nu.
\end{eqnarray}
By using theorem \ref{th12}, we have:
\begin{eqnarray}\label{duali}
    \chi(\mu)&=&\underset{f}{\inf}\int fd\mu+\eta(f),
\end{eqnarray}
To prove the first implication, we take $f,g$ to be continuous such that $f(x)+g(y)\geq xy$, and we take $\mu:=\nu_{f}$ and $\nu:=\nu_{g}$ to be free Gibbs measures associated with the potentials $f$ and $g$ respectively, and such that $\int xd\nu_{f}(x)=0$.
\bigbreak
The measures $\mu,\nu$ are well defined, since $f,g$ are continuous (a fortiori l.s.c.) and grow faster than any linear function, so the condition \eqref{gibbs} is satisfied.
\bigbreak
If we now apply the Talagrand inequality \eqref{ssfti2}, we get
\begin{equation}
    \underset{\pi}{\sup}\int x\cdot y d\pi \geq \int fd\mu+\eta(f)+\int gd\nu+\eta(g)-\log(2\pi),\nonumber
\end{equation}
and by using \eqref{kanto} and \eqref{duali} we have 
\begin{equation}
    \int fd\mu+\int gd\nu\geq \int fd\mu+\eta(f)+\int gd\nu+\eta(g)-\log(2\pi),\nonumber
\end{equation}
which (after removing the same terms on both sides and applying the exponential) is equivalent to
\begin{equation}
    e^{\eta(f)+\eta(g)}\leq 2\pi,\nonumber
\end{equation}
and, we get the desired conclusion.
\end{proof} 
\qed
\begin{remark}\label{rmk9}
    However, it is still unclear whether the converse holds in full generality: Does the free Blaschke-Santaló functional inequality imply (SSFTI) in its maximal correlation formulation? 
    \newline
    For now, we can explain one possible way to prove this statement, and the main obstacle to achieving the full proof.
    \bigbreak
    Given $\mu$, a centered compactly supported probability measure, and $\nu$, another compactly supported probability measure, and given two continuous functions $f$ and $g$ satisfying $f(x)+g(y)\geq xy$, our goal is to prove that
    \begin{eqnarray}
    \int fd\mu+\int gd\nu
    &\geq& \chi(\mu)+\chi(\nu)-\log(2\pi).\nonumber
    \end{eqnarray}
    To use the free Santal\'o inequality, the free Gibbs measure associated with $f$ must be centered: $\int xd\nu_f(x)=0$, and this is not the case at all.  
    \begin{flushleft}
        Let's \textbf{assume} for now that there is $\lambda \in \mathbb{R}$ such that
\begin{equation}\label{cent1}
    \int xd\nu_{f+\lambda Id}(x)=0.
\end{equation}
and note that since $\mu$ is centered, adding a linear function to it doesn't change the value of $\int fd\mu$, i.e. for all $\lambda\in \mathbb{R}$ we have
    \begin{equation}
        \int fd\mu=\int (f+\lambda id)d\mu.
    \end{equation}
\begin{flushleft}
We denote (for this choice of $\lambda$), $\tilde{f}(x):=f(x)+\lambda x$ and $\tilde{g}(y):=g(y+\lambda)$, for which we can easily check that $\tilde{f}(x)+\tilde{g}(y)\geq xy$ for all $(x,y)\in \mathbb{R}^2$.
\end{flushleft}
\begin{flushleft}
    So, if the assumption \eqref{cent1} is true, we get by using the free functional Santaló inequality,
\begin{equation}
    \eta(\tilde f)+\eta(\tilde g)\leq log(2\pi),\nonumber
\end{equation}
Now we have that $\int fd\mu=\int \tilde f d\mu$, since $\mu$ should be centered, and $\eta(\tilde g)=\eta(g)$ (this is as said before easily seen by using the limiting formula \eqref{freepress}, since the Lebesgue measure is invariant by translations). Therefore, adding on both sides $\int fd\mu$ and $\int gd\nu$, and using the dual formulation of the free entropy, we finally get
\begin{eqnarray}
    \int fd\mu+\int gd\nu&\geq &\underbrace{\int \tilde{f}d\mu+\eta(\tilde f)}_{\geq \chi(\mu)}+\underbrace{\int gd\nu+\eta(g)}_{\geq \chi(\nu)}-\log(2\pi)\nonumber\\
    &\geq& \chi(\mu)+\chi(\nu)-\log(2\pi).\nonumber
\end{eqnarray}
The proof would be completed by taking the supremum over all $f$ and $g$ on the left.
\end{flushleft}
So the main obstacle is to prove that there are $\lambda\in \mathbb{R}$ such that
\begin{equation}
    \int x d\nu_{f+\lambda id}(x)=0.
\end{equation}
A first idea we tried is to look at the following partition function,
\begin{eqnarray}
    Z(\lambda)&:=&\underset{n\rightarrow \infty}{\lim}\int \tr_N(M) \exp(-n^2(\tr_N(f(M))+\lambda \tr_N(M)))d\Lambda_N(M),\: \lambda \in \mathbb{R}\\\nonumber\
\end{eqnarray}
 which can be seen as the point limit of the sequence of functions $(Z_n)_{n\geq 1}$:
\begin{equation}
    Z_n(\lambda)=\int \tr_N(M) \exp(-n^2(\tr_N(f(M))+\lambda \tr_N(M)))d\Lambda_N(M),\: \lambda \in \mathbb{R},\nonumber
\end{equation}
Then we can check (using Fathi's arguments in dimension $d=n^2$ \cite{FathT}) that for all $n\geq 1$, $Z_n$ is well-defined, monotone (in $\lambda$) and smooth, since $f$ goes to infinity faster than any linear function, so its domain is an unbounded convex of $\mathbb{R}$, so its range is the whole space $\mathbb{R}$. In particular, there exists $(\lambda_n)\in \mathbb{R}^{\mathbb{N}},\: Z_n(\lambda_n)=0$. 
However, it is not sufficient to conclude that the result holds in the limit \cite{FathT}, i.e. there exists $\lambda \in\mathbb{R}$, $\int xd\nu_{f+\lambda id}(x)=0$, and this is a consequence of $(\lambda_n)_{n\geq 0}$ not being constant.
    \end{flushleft}
\end{remark}
\begin{flushleft}
    So we ask the following question:
\begin{conjecture}
    For every $f:\mathbb{R}\rightarrow \mathbb{R}$ continuous satisfying assumptions \eqref{gibbs}. Does $\lambda\in \mathbb{R}$ exist such that the free Gibbs measure associated with $f+\lambda id$ has barycenter zero?
\end{conjecture}
\end{flushleft}

\section{A inverse form of the free Blaschke-Santaló inequality}

In the context of convex geometry, the notion of \textit{Mahler volume} plays an important role in this study of convex bodies. For example, this quantity has nice properties: it is a continuous mapping if we equip the space of convex bodies in $\mathbb{R}^d$ with the usual Haussdorff topology. There are many connections between the study of this volume and other open problems in convex geometry, e.g. to the Bourgain slicing problem studied by Klartag \cite{klam}, which was a longstanding difficult conjecture finally solved in a recent work by Klartag and Lehec in \cite{klaleh}. This \textit{Mahler volume} is also an affine invariant, and it is known that the shapes with maximum Mahler volume are the ellipsoids (this is the case of equality in the Blaschke-Santal\'o inequality). However, the shape with the smallest value seems to be the hypercube (whose polar set is exactly the symmetric unit cube, generally called the octahedron). This was proved, at least asymptotically, in a pioneering work by Bourgain and Milman in \cite{a1} using methods from the geometry of Banach spaces and, more precisely, from the theory of cotypes. Obtaining these (non-asymptotic) lower bounds is in fact still an open problem, precisely formulated by Mahler \cite{mah} through the following famous conjecture in convex geometry (historically he formulated only the general case, not the symmetric one):
\begin{conjecture}(Mahler's Conjecture) \label{mahler}
    Let $K$ be a centrally symmetric convex body in $\mathbb{R}^d$, then the following bound holds:
    \begin{equation}
        \mathcal{P}(K):=\Vol_d(K)\Vol_d(K^{\circ})\geq \mathcal{P}(B_{\infty}^d)=\frac{4^d}{d!},
    \end{equation} where $B^d_{\infty}:=[-1,1]^d$ and equality holds if an if only $K$ is an Hanner polytope.
    \begin{flushleft}
        More generally for a general convex body $K$,
    \begin{equation}
        \mathcal{P}(K):=\underset{z\in K}{\min}\Vol_d(K)\Vol_d((K-z)^{\circ})\geq \mathcal{P}(\Delta_n)=\frac{(d+1)^{d+1}}{(d!)^2},
    \end{equation}
   where $\Delta_d$ is a simplex of $\mathbb{R}^d$ with equality if and if only $K=\Delta_d$.
    \end{flushleft}
    
\end{conjecture}
\bigbreak
The lower bound in the centrally symmetric case with order ${(d\:!)}^{-2}$ is known since Mahler work \cite{mah} and the \textit{Mahler conjecture} is now known for unconditional convex bodies (symmetric w.r.t all the hyperplanes $\left\{x_i=0\right\},\: i=1,\ldots,d$) as proved by Saint Raymond \cite{SaintR} and in small dimensions $n=2,3$. The third dimensional case being recently proved in a important work of Iriyeh and Shibata \cite{symM}.
\bigbreak
\begin{flushleft}
Like the Blaschke-Santaló inequality, the Mahler conjecture admits a functional form discovered by Fradelizi and Meyer \cite{FraMey}, which however is stronger than the Mahler conjecture for sets, since for the functional version to hold the geometric conjecture must hold in every dimension.
\end{flushleft}
\begin{theorem}(Fradelizi-Meyer 2008 \cite{FraMey}).
The symmetric Mahler conjecture \eqref{mahler} holds if and only if for all $f$ even, convex functions.
    \begin{equation}
    \log \int e^{-f}dx+\log \int e^{-f^*}dx\geq d\log 4,
\end{equation}
where $f^*$ denote the Legendre transform.
\begin{flushleft}
Moreover the conjecture is true for $f$ convex, unconditional, i.e. $f(x_1,\ldots,x_d)=f(\lvert x_1\rvert,\ldots,\lvert x_d\rvert)$. 
\end{flushleft}
\end{theorem}
We can indeed prove the following proposition by reasoning as in Proposition $3.5$ in \cite{BB} or by using the variational characterisation \eqref{gibbs}.
\begin{proposition}
    Let $f:\mathbb{R}\rightarrow \mathbb{R}$ be an even differentiable function satisfying \eqref{gibbs}, then $\nu_f$ is a symmetric probability measure.
\end{proposition}
\begin{remark}
The differentiability conditions can be weakened by assuming that $f$ is differentiable on $\supp(\nu_f)$. In fact, from the Euler-Lagrange or Schwinger-Dyson equations, it's enough to look only inside the support of $\nu_f$, since it's a diffuse measure.
    
\end{remark}
\begin{flushleft}
 The following theorem is obtained by a random matrix approximation. Since the ideas are the same as for our previous free version of the free Blaschke-Santaló inequality \eqref{san}, here with an even potential (so the barycenter of the limiting free Gibbs measure exists and is always zero, which we don't need to check), we leave the details to the reader; the analytical proof is much more difficult, and this is an absolutely non-trivial statement: It has the same complexity as in the classical case as it was already mentioned by Fradelizi and Meyer in the remark following propositions $1$ and $2$ \cite{FraMey2}. Therefore, we think it deserves to be presented in a separate full paper. 
\end{flushleft}
\begin{theorem}(Free inverse Blaschke-Santaló inequality).\label{inverse11}
    Let $f$ be a l.s.c convex even function such that $\underset{\lvert x\rvert\rightarrow \infty}{\lim}f(x)=+\infty$ (or at least having a minimum), then
    \begin{equation}\label{inverse12}
        \chi_f(\nu_f)+\chi_{f^*}(\nu_{f^*})\geq \log(4),
    \end{equation}
    or equivalently with the notations introduced before, we have
    \begin{equation}
        \eta(f)+\eta(f^*)\geq \log(4).
    \end{equation}
\end{theorem}
\begin{remark}
    This inequality says in fact that the minimum (not unique) is obtained for the family of arcsine laws, which are the laws supported on $[-R,R]$ for some $R>0$ and of density
    \begin{equation*}
        \frac{1}{\pi\sqrt{R-x^2}}.
    \end{equation*}
For example, if $R=1$, this free Gibbs measure is associated with the potential $V(x)=\log(2),\:x\in [-1,1]$ and 
   $V(x):=\frac{1}{\pi}\log\lvert \frac{\lvert x\rvert+\sqrt{x^2-1}}{2}\rvert\:,x\notin [-1,1]$.
    This is consistent with the result in the classical case, since in dimension one the Blaschke-Santaló function over  $\mathbb{R}$ is minimal for symmetric segments (which are the only Hanner polytopes in dimension one), and it is well known that the arcsine distribution is considered as the free analogue of the uniform distribution. Equivalently, it is easy to see that their "dual" is the Bernoulli law. In particular, it would be interesting to prove that the minimum is non-unique and can only be attained by these two families of laws.
    \end{remark}
\bigbreak
\begin{flushleft}
   Using the same scheme as in the proof of proposition \ref{prop8}, we are able to prove a inverse form of (SSFTI) for symmetric free Gibbs measures (e.g. for even potentials $u$ with at least quadratic growth).
    \begin{theorem}
        Let $\mu:=\nu_f$ a free Gibbs measure associated with a even l.s.c convex potential $f$ with $\underset{\lvert x\rvert\rightarrow \infty}{\lim}f(x)=+\infty$ (or at least with a minimum), and let $\mu^*=\nu_{f^*}$ denote the free Gibbs measure associated with $f^*$ (the Legendre transform of $f$). Then, we have
         \begin{equation}
        H(\mu,\sigma)+H(\mu^*,\sigma)\leq \frac{1}{2}W_2(\mu,\mu^*)^2+\frac{1}{2}\log (\pi/2).
        \end{equation}
    \end{theorem}
    \begin{proof}
    First, note that $\mu^{*}$ is well defined. In fact, by
       the properties of the Legendre transform \eqref{property}, we have that $f^*$ is l.s.c. and satisfies \eqref{gibbs}.
       \bigbreak
       Let $\mathcal{F}_{ev}$ be the set of even (unconditionally) l.s.c. convex functions. Then we have the following, which holds by a simple localization argument,
        \begin{equation}
            \underset{\pi}{\sup}\int x\cdot yd\pi=\underset{f\in \mathcal{F}_{ev}}{\inf}\int fd\mu+\int f^*d\mu^*,\nonumber
        \end{equation}
        where $\pi$ is a transport plan between the marginals $\mu$ and $\mu^*$. 
        \bigbreak
        This fact follows from the following arguments: it is well known from Kantorovitch duality that
        \begin{equation}\label{convK}
            \underset{\pi}{\sup}\int  x\cdot yd\pi=\underset{f:\:\mathbb{R}\rightarrow \mathbb{R}}{\inf}\int fd\mu+\int f^*d\mu^*;\nonumber
        \end{equation}
        where $f$ is a lower semi-continuous convex function (see e.g. particular case $5.3$ in \cite{vill}).
        \bigbreak
       Note that in this case $\mu$ (hence $\mu^*$) are symmetric measures. In fact, reasoning as in proposition $3.5$ in \cite{BB} to show that or using the variational characterization \eqref{gibbs}, it is not difficult to show that $\supp(\nu_u)$ must be a symmetric compact interval, and that $\nu_u(B)=\nu_u(-B)$ for any $B\subset \mathcal{B}(\mathbb{R})$ (which means nothing more than saying in dimension one that the distribution is symmetric with respect to the origin). This is even easier to deduce if the potential $f$ has enough regularity, e.g. if $f$ is of class $\mathcal{C}^3$ we have an explicit formula for the density of the free Gibbs measure, see e.g. Ledoux and Popescu, Theorem $1$, \cite{LP}.
       \bigbreak
       It is then easy to see that the optimizer also inherits the symmetry property of these two measures. So we can restrict the definition of the supremum in \eqref{convK} to the functions $f\in \mathcal{F}_{ev}$. 
\bigbreak
        We can also easily check that
        \begin{equation}
            \chi(\mu)+\chi(\mu^*)=\underset{f\in \mathcal{F}_{ev}}{\inf}\:\int fd\mu+\int f^*d\mu^*+\eta(f)+\eta(f^*),\nonumber
        \end{equation}
        This is trivial to get, since the optimizer is the function $f$ itself.
        \bigbreak
        Using these formulas, as well as the theorem \eqref{inverse11} and the previous approach used to prove the proposition \ref{prop8}, we finally get
        \begin{equation}
            \chi(\mu)+\chi(\mu^*)\geq \underset{\pi}{\sup}\int x\cdot yd\pi+\log 4,\nonumber
        \end{equation}
        Subtraction of (half of) the second moments of each measure and the constant $\log(2\pi)$ on each side finishes the proof.
    \end{proof} \qed
    \end{flushleft}

\section{A multidimensional version of SSFTI}\label{last}
\begin{flushleft}
 In this section we adopt the notations and conventions of Hiai and Ueda \cite{HiaiUeda} for stating a multidimensional extension of the previous theorem \ref{th8} for tracial non-commutative distributions. This part is again based on a random matrix approximation procedure.
 \end{flushleft}
\begin{flushleft}
 If $\mathcal{A}$ is a unital $C^*$-algebra, $\mathcal{A}^{sa}$ stands for the set of self-adjoint elements of $\mathcal{A}$, and we denote $S(A)$ by the state space of $A$ and its restriction to tracial states $TS(A)$ (recall that not every state is tracial, e.g. in a type $III$ von Neumann algebra). The universal free product $C^*$-algebra of two copies of $\mathcal{A}$ is denoted by $\mathcal{A}\star \mathcal{A}$, and $\sigma_1$ and $\sigma_2$ stand for the canonical embedding maps of $\mathcal{A}$ into the left and right copies of $\mathcal{A}$.
 \end{flushleft}
\begin{flushleft}
Let us fix $n\in \mathbb{N}$ and $R>0$. We denote the universal free product $C^*$-algebra $\mathcal{A}_{R}^{(n)}:=\mathcal{C}([-R,R])^{\star n}$ with norm $\lVert \cdot\rVert_{R}$ and a canonical set of self-adjoint generators $X_i(t)=t$ in the $i$th copy of $\mathcal{C}([-R,R]),\:1\leq i\leq n$. Each $\varphi\in S(\mathcal{A}_{R}^{(n)})$ provides a (non-commutative) \textit{distribution} or \textit{law} of $X_1,\ldots,X_n$ whose (non-commutative moments) are given by $\varphi(X_{i_1}\ldots,X_{i_m}),\: 1\leq i_1,\ldots,i_m\leq n$.
Indeed, any distribution in the $C^*$ algebra setting can be realised in this way. More precisely, if $a_1,\ldots,a_n$ are self
adjoint variables in a $C^*$-probability space $(\mathcal{A},\phi)$ with operator norm $\lVert a_i\rVert<R$, then one has a (unique) $*$-homomorphism $\psi$ from into $\mathcal{A}_{R}^{(n)}$ sending any $X_i$ to $a_i$ so that the distribution of $X_1, \ldots , X_n$ under $\phi\circ \psi \in S(\mathcal{A}_{R}^{(n)})$ coincides with that of $a_1, \ldots a_n$ under $\phi$.
\bigbreak
Since we will be dealing with non-commutative Wasserstein distance and free entropy, which are well understood and have good properties only for tracial states, we will consider only tracial distributions, i.e. elements in $TS(\mathcal{A}_{R}^{(n)})$.

  \begin{flushleft}
        Recall that every probability measure $\lambda \in \mathcal{P}((M_N^{sa})^n)$ gives rise to the tracial distribution $\hat{\lambda}_R\in TS(\mathcal{A}_R^{(n)})$ which is defined by:
        \begin{equation}
            \hat{\lambda}_R(h):=\int_{(M_N^{sa})^n}\frac{1}{N}\Tr_N(h(r_R(M_1),\ldots,r_R(M_n)))d\lambda(M_1,\ldots,M_n),\:\:\: h\in \mathcal{A}_R^{(n)}.\nonumber
        \end{equation}
        where $r_R$ denotes the retraction on $\mathbb{R}$ onto $[-R,R]$.

\begin{definition}
    We will say that $\lambda\in \mathcal{P}((M_N^{sa})^n)$ is centered if it has first absolute moment and if for all $i=1,\ldots,n$,
    \begin{equation}
        \int_{(M_N^{sa})^n} \frac{1}{N}\Tr_N(M_i)d\lambda(M_1,\ldots,M_n)=0.\nonumber
    \end{equation}
\end{definition}
We will also say that $\tau \in TS(\mathcal{A}_R^{(n)})$ is centered if for all $i=1,\ldots,n$, we have
\begin{equation}
    \tau(X_i)=0.\nonumber
\end{equation}
    \end{flushleft}
\begin{flushleft}
Let $\pi_{\tau}$ be the GNS representation of $\mathcal{A}_{R}^{(n)}$ associated with $\tau$, and denote $((\pi(\mathcal{A}_{R}^{(n)}))'',\tilde{\tau})$ where $\tilde{\tau}$ is the normal extension of $\tau$ with the self-adjoint variables $\pi_{\tau}(X_1),\ldots,\pi_{\tau}(X_n)$.
\bigbreak
    The microstates \textit{free entropy} $\chi$ introduced by Voiculescu \cite{V} is defined in our context as as follows
    \begin{equation}
        \chi(\tau):=\lim_{\substack{m\rightarrow \infty\\ \varepsilon\searrow 0}}\limsup_{N\rightarrow \infty}\bigg(\frac{1}{N^2}\log \Lambda_N^{\otimes n}(\Gamma_R(\pi_{\tau}(X_1),\ldots,\pi_{\tau}(X_n);N,m,\varepsilon))+\frac{n}{2}\log 2\pi\bigg),\nonumber
    \end{equation}
    where $\Gamma_R(\pi_{\tau}(X_1),\ldots,\pi_{\tau}(X_n);N,m,\varepsilon)$ denotes the approximating microstates which are $n$-tuples $(A_1,\ldots,A_n)\in (M_{N,R}^{sa})^n$ (tuples of self-adjoint matrices such that $\underset{i=1,\ldots,n}{\max} \lVert A_i\rVert\leq R$) such that
    \begin{equation}
        \bigg\lvert\tau(X_{i_1}\ldots X_{i_p})-\frac{1}{N}\Tr_N(A_{i_1}\ldots A_{i_p})\bigg\rvert=\bigg\lvert\tilde{\tau}(\pi_{\tau}(X_{i_1})\ldots \pi_{\tau}(X_{i_p}))-\frac{1}{N}\Tr_N(A_{i_1}\ldots A_{i_p})\bigg\rvert<\varepsilon\nonumber
    \end{equation}
    for all possible choice of $i_1,\ldots,i_p,\: 1\leq p\leq m$.
    \newline
    For sake of convenience, we can choose a sub-sequence $N(1)<N(2)<\ldots$ in such a way that letting 
    \begin{equation}
        \Gamma_R(\tau;k):=\Gamma_R(\pi_{\tau}(X_1),\ldots,\pi_{\tau}(X_n);N,k,\frac{1}{k})\nonumber
    \end{equation}
  
    we have
    \begin{equation}
        \chi(\tau)=\lim_{k\rightarrow \infty}\bigg(\frac{1}{N(K)^2}\log \Lambda_N^{\otimes n}(\Gamma_R(\tau;k))+\frac{n}{2}\log 2\pi\bigg).
    \end{equation}
\end{flushleft}
\begin{flushleft}
    Following what we defined in the previous section, we denote the following Gaussian product measure on $\mathcal{P}((M_N^{sa})^n)$,
    \begin{equation}
        \sigma_{N}^{(n)}:=\bigotimes_{i=1}^n \sigma_{N} \in \mathcal{P}((M_N^{sa})^n).
    \end{equation}
    That is \begin{equation}
        d\sigma_{N}^{(n)}(M_1,\ldots,M_n)=\frac{1}{Z_N^n}\exp\bigg[-N\bigg(\frac{1}{2}\sum_{i=1}^nTr(M_i^2)\bigg)\bigg]d\Lambda_N^{\otimes n}(M_1,\ldots,M_n).\nonumber
    \end{equation}
    We then have by the celebrated asymptotic \textit{freeness} result of Voiculescu \cite{Vfree}.
    \begin{lemma}
        For all $R\geq 2$, we have
    \begin{equation}
        \underset{N\rightarrow \infty}{\lim} \widehat{\sigma}_{N,R}^{(n)}=\tau_{\frac{1}{2}\sum_{i=1}^n X_i^2}=\sigma^{\star n}\:\:\: \mbox{weakly}^*.
    \end{equation}
    \end{lemma}
\end{flushleft}
\begin{definition}(Biane Voiculescu, \cite{BV})
The free quadratic Wasserstein distance is defined as
\begin{eqnarray*}
    W_2(\tau_1,\tau_2)^2&=&\inf \Bigl\{\sum_{i=1}\tau(\lvert \sigma_1(a_i)-\sigma_2(a_i)\rvert^2\:\tau\in TS(\mathcal{A}\star\mathcal{A}|\tau_1,\tau_2)\Bigl\},\end{eqnarray*}
where \begin{equation}
    TS(\mathcal{A}\star\mathcal{A}|\tau_1,\tau_2):=\left\{\tau\in TS(\mathcal{A}\star\mathcal{A}): \tau\circ \sigma_1=\tau_1,\tau\circ\sigma_2=\tau_2\right\}.\nonumber
\end{equation}
\end{definition}
\begin{remark}
    We recall the following properties satisfied by this free quadratic Wasserstein distance (see \cite{HiaiUeda}, Proposition $1.2$ for a more general statement): for every $\tau, \tau'\in TS(\mathcal{A})$ where $\mathcal{A}$ is a unital $C^*$-algebra, we have 
    \begin{enumerate}
        \item $W_2(\tau,\tau')$ is a metric on $TS(A)$.
        \item $W_2(\tau,\tau')$ is jointly lower semi-continuous in $(\tau,\tau')\in TS(\mathcal{A})\times TS(\mathcal{A})$ in the weak$^*$-topology (see \cite{BV}, Proposition $1.4$).
        \item $W_2(\tau,\tau')$ is jointly convex in $(\tau,\tau')\in TS(\mathcal{A})\times TS(\mathcal{A})$.
    \end{enumerate}
\end{remark}

Let us recall the following lemma proved by Hiai and Ueda (see lemma $1.3$ in \cite{HiaiUeda}), which gives an inequality comparing the free quadratic Wasserstein distance between two non-commutative distributions with the usual quadratic Wasserstein distance for their random matrix distributions.
\begin{lemma}\label{lma1}
    For each pair $\lambda_1,\lambda_2\in \mathcal{P}((M_N^{sa})^n)$ and each $R>0$, let $\widehat{\lambda}_{1,R},\widehat{\lambda}_{2,R}\in TS(\mathcal{A}_R^n)$ be the corresponding random matrix distributions. Then we have
    \begin{equation}
        W_2(\widehat{\lambda}_{1,R},\widehat{\lambda}_{2,R})\leq\frac{1}{\sqrt{N}}W_2(\lambda_1,\lambda_2),
    \end{equation}
    where $W_2(\lambda_1,\lambda_2)$ is the usual quadratic Wasserstein distance between $\lambda_1,\lambda_2$, defined in this case by
    \begin{equation}
        W_2(\lambda_1,\lambda_2)^2:=\underset{\pi}{\inf}\iint_{(M_N^{sa})^n\times (M_N^{sa})^n} \sum_{i=1}^n\lVert A_i-B_i\rVert_{HS}^2d\pi.
    \end{equation}
    where $\pi$ denotes the set of transport plans of marginals $\lambda_1,\lambda_2$.
\end{lemma}

\begin{flushleft}
    Before introducing the main theorem, we need to set the multivariate version of the free entropy adapted to the free Ornstein-Uhlenbeck process. In fact, we set for every $\tau\in TS(\mathcal{A}_R^{(n)})$
\begin{equation}
    H(\tau,\sigma^{\star n}):=2\bigg(-\chi(\tau)+\tau\bigg(\frac{1}{2}\sum_{i=1}^nX_i^2\bigg)+\frac{n}{2}\log 2\pi\bigg).
\end{equation}
\end{flushleft}
We also need to introduce the following lemma, which basically states the (classical) sharp symmetrized Talagrand inequality at the level of random matrices. This will be used to derive the result for the free case by going to the limit and using the lower semi-continuity of the free quadratic Wassertein distance.
\begin{lemma}\label{lma2}
    For every $N\in \mathbb{N}^*$ and $\lambda,\lambda'\in \mathcal{P}((M_N^{sa})^n)$ with $\lambda$ centered, then we have:
    \begin{equation}
    W_2(\lambda,\lambda')^2\leq \frac{2}{N}\Ent_{\sigma_N}(\lambda)+\frac{2}{N}\Ent_{\sigma_N}(\lambda').
    \end{equation}
    \begin{proof}
        This is just an application of the corollary \ref{covgau} when considering $(M_N^{sa})^n$ as $\mathbb{R}^{N^2n}$ and, the use of the tensorization property of the classical sharp Talagrand inequality. Indeed, in this case, $\sigma_N$ is a product of independent centered Gaussian variables of variance $\frac{1}{N}$. Assuming now that each component of the $n$-tuple of random matrices $\lambda=([\lambda_{i}]_{1\leq j,k\leq N})_{i=1}^n\in (M_N^{s. a})^n:=((\lambda_{k,k}^i)_{1\leq k\leq N},(\sqrt{2}\Re(\lambda_{j,k}^i))_{j<k},(\sqrt{2}\Im(\lambda_{j,k}^i))_{j<k})_{i=1}^n$ is centered, concludes.
    \end{proof} 
    \qed
\end{lemma}
\begin{theorem}
    Let $R\geq 2$, and $\sigma^{\star n}\in TS(\mathcal{A}_R^n)$ as the non-commutative distribution of a standard semicircular system, then
    \begin{equation}
        W_2(\tau,\tau')^2\leq 2H(\tau,{\sigma^{\star n}})+2H(\tau,{\sigma^{\star n}})
    \end{equation}
    for every $\tau\in TS(\mathcal{A}_R^n)$ centered and every $\tau'\in TS(\mathcal{A}_R^n)$.
\end{theorem}

\begin{proof}
We assume without loss of generality that both $\chi(\tau),\chi(\tau')>-\infty$, as if not, there is nothing to prove.
\newline
Then, recall that we can choose a sub-sequence $N(1)<N(2)<\ldots$ in such a way that letting 
    \begin{equation}
        \Gamma_R(\tau;k):=\Gamma_R(\pi_{\tau}(X_1),\ldots,\pi_{\tau}(X_n);N,k,\frac{1}{k}),\nonumber
    \end{equation}
we have
    \begin{equation}
        \chi(\tau)=\chi(\pi_{\tau}(X_1),\ldots,\pi_{\tau}(X_n))=\lim_{k\rightarrow \infty}\bigg(\frac{1}{N(K)^2}\log \Lambda_N^{\otimes n}(\Gamma_R(\tau;k))+\frac{n}{2}\log 2\pi\bigg).\nonumber
    \end{equation}
Let us first introduce the following random matrix distribution $\hat{\lambda}_{N(k),R}\in \mathcal{P}((M_{N(K),R}^{sa})^n)$ and associated with the probability measure
\begin{equation}
    \lambda_{N(k)}:=\frac{1}{\Lambda_{N(k)}^{\otimes n}(\Gamma_R(\tau;k))}\Lambda_{N(k)}^{\otimes n}|_{\Gamma_R(\tau;k)}\in \mathcal{P}((M_{N(K),R}^{sa})^n)\nonumber
\end{equation}
It is then shown in \cite{HiaiUeda} (see the proof of theorem $2.2$) that 
\begin{equation}\label{conver}
    \underset{k\rightarrow\infty}{\lim} \hat{\lambda}_{N(k),R}=\tau,\:\:\: \mbox{weakly}^*.\nonumber
\end{equation}
By using Lemmas \ref{lma1} and \ref{lma2}, we have
\begin{equation}
    W_2(\widehat{\lambda}_{N(k),R}, \widehat{\lambda'}_{N(k),R})^2\leq \frac{2}{N(k)^2}\Ent_{\sigma_{N(k)}}(\lambda_{N(k)})+\frac{2}{N(k)^2}\Ent_{\sigma_{N(k)}}(\lambda'_{N(k)})\nonumber
\end{equation}
where both $\frac{1}{N(k)^2}\Ent_{\sigma_{N(k)}}(\lambda_{N(k)})$ and $\frac{1}{N(k)^2}\Ent_{\sigma_{N(k)}}(\lambda'_{N(k)})$ converges respectively to $H(\tau,\sigma^{\star n})$ and $H(\tau',\sigma^{\star n})$ when $k\rightarrow \infty$ (see details in the proof of theorem $2.2$ in \cite{HiaiUeda}).
\bigbreak
\begin{flushleft}
    And finally by using the joint lower semi-continuity of free quadratic Wasserstein distance $W_2$, Lemma \ref{lma1} and \eqref{conver} we get
    \begin{equation}
        W_2(\tau,\tau')^2\leq \liminf_{k\rightarrow \infty}W_2(\widehat{\lambda}_{N(k),R}, \widehat{\lambda'}_{N(k),R})^2.\nonumber
    \end{equation}
    which completes the proof..
\end{flushleft}
\end{proof} \qed

\end{flushleft}
\begin{remark}
It is worth noting that this theorem also holds if the free entropy is replaced by a quantity similar to the free entropy $\eta$: the multidimensional free pressure (introduced by Hiai \cite{Hiai2}), which is defined as the Legendre transform of the free entropy. This comes at the cost of a slight weakening of the hypothesis, since the non-commutative tracial distributions under consideration must arise as an equilibrium tracial state for some $h\in (\mathcal{A}_R^{(n)})^{sa}$, and gives a sharper bound, since it is known that $\chi(\tau)\leq \eta(\tau)$ by Proposition $4.5$ in \cite{Hiai2}. However, this result would only be interesting in the multidimensional case (since these two entropies coincide in dimension one). Another topic of great interest would be to obtain such a result for the non-microstates version of the free entropy $\chi^*$; this problem seems to be much more difficult, as can be seen in the proof of Dabrowski's Talagrand transport cost inequality for the non-microstates free entropy (Theorem 26 in \cite{Dab10}), which is based on obtaining an infinitesimal estimate in the free quadratic Wasserstein distance between two close points of a free Ornstein-Uhlenbeck process. This is quite difficult and requires a deep understanding of certain free Markov dilations, which are (embeddings) deformations of a finite von Neumann algebra, e.g. $\mathcal{M}=W^*(X_1,\ldots,X_n)$ with $X_1,\ldots,X_n$ in a finite von Neumann algebra and its associated Markov semigroup $\phi_t=e^{-\Delta t}$ with the $L^2$-generator $\Delta$ in divergence form, i.e. $\Delta=\delta^*\bar{\delta}$, where $\delta$ denotes the free difference quotient.  In fact, to get this kind of dilation, the free difference quotient must have better properties than just being closable (finite free Fisher information $\Phi^*(X)<\infty$), e.g. with Lipschitz conjugate variables \cite{Dab10,DabIoa}, which is a stronger assumption, since in this case $W^*(X_1,\ldots,X_n)$ has many structural properties similar to those of the free group factor $L(\mathbb{F}_n)$.
\bigskip
\begin{flushleft}
 Using the same procedure and the theorem \ref{thfathi} of Fathi \cite{FathT}, it is also not difficult to see that this sharp symmetrized free Talagrand inequalities can be extended to more general free Gibbs laws coming from the free product of a one-dimensional free Gibbs measure associated with strictly convex potentials: that is, if $V(X_1,\ldots,X_n)=V(X_1)+\ldots+V(X_n), \: \mbox{with},\:\forall i=1,\ldots, n\:, V_i''\geq c>0, \: \tau_{V_i}\sim \nu_{V_i}$ and such that $\tau_V=\tau_{V_1}*\ldots*\tau_{V_n}$. 
 One can even find a pure mass transport proof (under a slightly stronger regularity assumption) to extend the sharp symmetrized free Talagrand inequality to a free (one-dimensional) Gibbs measure associated with a uniformly convex potential. This is a consequence of a one-dimensional free version of the Cafarelli contraction theorem, recently proved by the author in \cite{CD_mm}. For the sake of brevity, we leave the details to the reader.
 \end{flushleft}
 \end{remark}

 \section{Open Questions}

    In the classical context of convex geometry, there are numerous results (such as those mentioned in this paper) that can be translated from a geometric formulation to a functional version and vice versa. This has been a topic of profound interest in geometric functional analysis, e.g. the Prékopa-Leindler inequality as a strong version of the Brunn-Minkowski inequality, the functional Santal\'o inequality of Ball in his thesis \cite{Ball}, or even the \textit{Mahler conjecture} of Fradelizi and Meyer \cite{FraMey}, the entropic form of the hyperplane slicing conjecture . In this work we have been able to derive new free functional inequalities with a convex geometry flavor (the list presented here is not exhaustive, as many variants of the inequalities proposed here are easy to obtain). We wonder whether this kind of geometric aspect makes sense in the free context (and hopefully in the multidimensional case), as it could be very useful to have a better understanding of what \textit{free (convex) geometry} really is. There are probably also many analogous conjectures in the free case (some of which remain unproven in the classical case). In particular, since the semicircular law maximizes the free Blaschke-Santal\'o functional and the Bernoulli law (or by duality the arcsine law) minimizes it (conjecturally true in a multidimensional setting) and, since von Neumann algebras generated by semicircular operators are free group factors and in the Bernoulli case we obtain a hyperfinite algebra (at the level of the $C^*$-algebra we have an AF-algebra), we wonder whether this free Blaschke-Santal\'o inequality and the cases of equality can be reformulated in a more geometrical way, i.e. in terms of algebras (probably first for $*$-algebras) and conjecturally for von Neumann algebras, which in our context are the free analogue of convex bodies and which reinforce this phenomenon discovered in \cite{CD_mm}.
\section{Acknowledgments}
This work was partially supported by the Luxembourg National Research Fund 
\newline
(Grant: O22/17372844/FraMStA). CP Diez also acknowledges support from the Ministry of Research, Innovation and Digitalization (Romania), grant CF-194-PNRR-III-C9-2023. 
\newline
The author would like to thank Prof. Max Fathi, Prof. Ivan Nourdin, Prof. Giovanni Peccati and Dr Pierre Perruchaud, for their insightful comments, references and suggestions on this paper.

\end{document}